\documentclass[a4paper,11pt]{article}

\usepackage{pdfpages}
\usepackage{graphicx}
\usepackage{mathtools}
\usepackage{verbatim}
\usepackage{amssymb}
\usepackage{bm}
\usepackage{enumitem}
\usepackage{cases}
\usepackage{latexsym}
\usepackage{amsthm}
\usepackage{amsfonts}
\usepackage{bbm}
\usepackage{etoolbox}
\usepackage[top=25mm, bottom=25mm, left=28mm, right=28mm]{geometry}
\usepackage{cite}
\usepackage{changepage}
\usepackage{fancybox}
\usepackage{tikz}

\newtheorem{theorem}{Theorem}[section]
\AfterEndEnvironment{theorem}{\noindent\ignorespaces}
\newtheorem{corollary}[theorem]{Corollary}
\AfterEndEnvironment{corollary}{\noindent\ignorespaces}
\newtheorem{lemma}[theorem]{Lemma}
\AfterEndEnvironment{lemma}{\noindent\ignorespaces}
\newtheorem{proposition}[theorem]{Proposition}
\AfterEndEnvironment{proposition}{\noindent\ignorespaces}

\AfterEndEnvironment{question}{\noindent\ignorespaces}
\newtheorem{conjecture}[theorem]{Conjecture}
\AfterEndEnvironment{conjecture}{\noindent\ignorespaces}
\newtheorem{problem}[theorem]{Problem}
\AfterEndEnvironment{problem}{\noindent\ignorespaces}

\theoremstyle{definition}
\newtheorem{definition}[theorem]{Definition}
\AfterEndEnvironment{definition}{\noindent\ignorespaces}

\theoremstyle{remark}

\AfterEndEnvironment{example}{\noindent\ignorespaces}

\AfterEndEnvironment{remark}{\noindent\ignorespaces}

\AfterEndEnvironment{notation}{\noindent\ignorespaces}

\newenvironment{proof2}{\proof[\textnormal{\textbf{Proof.}}]}{\qed}

\usepackage{chngcntr}
\usepackage{apptools}
\AtAppendix{\counterwithin{theorem}{section}}

\def\e{\varepsilon}

\def\N{\mathbb{N}}

\def\a{\alpha}

\def\d{\delta}

\def\G{\mathcal{G}}
\def\L{\mathcal{L}}

\def\s{\subset}
\def\ex{\mathrm{ex}}

\begin{document}

\title{More on the extremal number of subdivisions}

\author{David Conlon\thanks{Department of Mathematics, California Institute of Technology, USA.
E-mail: {\tt
dconlon@caltech.edu}. Research supported by ERC Starting Grant RanDM 676632.} \and 
Oliver Janzer\thanks{Department of Pure Mathematics and Mathematical Statistics, University of Cambridge, United Kingdom.
E-mail: {\tt oj224@cam.ac.uk}.} \and
Joonkyung Lee\thanks{Fachbereich Mathematik, Universit\"at Hamburg, Germany.
E-mail: {\tt
joonkyung.lee@uni-hamburg.de}. Research supported by ERC Consolidator Grant PEPCo 724903.}}

\date{}

\maketitle

\begin{abstract}
Given a graph $H$, the extremal number $\ex(n,H)$ is the largest number of edges in an $H$-free graph on $n$ vertices. We make progress on a number of conjectures about the extremal number of bipartite graphs. First, writing $K'_{s,t}$ for the subdivision of the bipartite graph $K_{s,t}$, we show that $\ex(n, K'_{s,t}) = O(n^{3/2 - \frac{1}{2s}})$. This proves a conjecture of Kang, Kim and Liu and is tight up to the implied constant for $t$ sufficiently large in terms of $s$. Second, for any integers $s, k \geq 1$, we show that $\ex(n, L) = \Theta(n^{1 + \frac{s}{sk+1}})$ for a particular graph $L$ depending on $s$ and $k$, answering another question of Kang, Kim and Liu. This result touches upon an old conjecture of Erd\H{o}s and Simonovits, which asserts that every rational number $r \in (1,2)$ is realisable in the sense that 
$\ex(n,H) = \Theta(n^r)$ for some appropriate graph $H$, giving infinitely many new realisable exponents and implying that $1 + 1/k$ is a limit point of realisable exponents for all $k \geq 1$. Writing $H^k$ for the $k$-subdivision of a graph $H$, this result also implies that for any bipartite graph $H$ and any $k$, there exists $\delta > 0$ such that $\ex(n,H^{k-1}) = O(n^{1 + 1/k - \delta})$, partially resolving a question of Conlon and Lee. Third, 
extending a recent result of Conlon and Lee, we show that any bipartite graph $H$ with maximum degree $r$ on one side which does not contain $C_4$ as a subgraph satisfies $\ex(n, H) = o(n^{2 - 1/r})$. 
\end{abstract}

\section{Introduction} \label{sectionintroturan}

For a graph $H$, the {\it extremal number} $\ex(n,H)$ is the maximal number of edges in an $H$-free graph on $n$ vertices. The celebrated Erd\H os--Stone--Simonovits theorem \cite{ES46,ESi66} states that $\ex(n,H)=\left(1-\frac{1}{\chi(H)-1}+o(1)\right)\frac{n^2}{2}$, where $\chi(H)$ is the chromatic number of $H$. This determines the asymptotics of $\ex(n,H)$ for any $H$ of chromatic number at least 3. However, for bipartite graphs $H$, it only gives $\ex(n,H)=o(n^2)$. One of the central problems in extremal combinatorics is to obtain more precise bounds in this case. For an overview of this interesting area, we refer the reader to the comprehensive survey by F\"uredi and Simonovits~\cite{FS13}.

Our starting point here lies with one of the few general results in the area, first proved by F\"uredi~\cite{Fu91} and later reproved by Alon, Krivelevich and Sudakov~\cite{AKS03} using the celebrated dependent random choice technique~\cite{FS11}. Note that here and throughout, we use the asymptotic notation $O,o,\Omega,\omega$ to indicate that $n\rightarrow \infty$ and everything else is kept constant. In particular, the implied constants for $O$ and $\Omega$ can depend on any parameter other than $n$.

\begin{theorem}[F\"uredi, Alon--Krivelevich--Sudakov] \label{faks}
	Let $H$ be a bipartite graph such that in one of the parts all the degrees are at most $r$. Then $\ex(n,H) = O(n^{2-1/r})$.
\end{theorem}

This result is known to be tight~\cite{KRS96, ARSz99}, since, for $s$ sufficiently large in terms of $r$, $\ex(n,K_{r,s})=\Omega(n^{2-1/r})$. Moreover, it is conjectured~\cite{KST54} that this should already hold when $s = r$. On the other hand, a recent conjecture of Conlon and Lee~\cite{CL18} says that containing $K_{r,r}$ as a subgraph should be the only reason why Theorem~\ref{faks} is tight up to the constant.

\begin{conjecture}[Conlon--Lee] \label{conjecturefaks}
	Let $H$ be a bipartite graph such that in one of the parts all the degrees are at most $r$ and $H$ does not contain $K_{r,r}$ as a subgraph. Then there exists some $\d>0$ such that $\ex(n,H)=O(n^{2-1/r-\d})$.
\end{conjecture}

To say more, recall that the {\it $k$-subdivision} of a graph $L$ is the graph obtained by replacing the edges of $L$ by internally disjoint paths of length $k+1$. We shall write $L^k$ for the $k$-subdivision of $L$ and $L'$ for the 1-subdivision. It is easy to see that any $C_4$-free bipartite graph in which every vertex in one part has degree at most two is a subgraph of $K'_t$ for some positive integer $t$. Conlon and Lee~\cite{CL18} verified their conjecture in the $r=2$ case by proving the following result.

\begin{theorem}[Conlon--Lee] \label{clsub}
	For any integer $t\geq 3$, $\ex(n,K'_t)=O(n^{3/2-1/6^t})$.
\end{theorem}

Our first result gives some small progress towards Conjecture~\ref{conjecturefaks} when $r > 2$.

\begin{theorem} \label{maxdegr}
	Let $H$ be a bipartite graph such that in one of the parts all the degrees are at most $r$ and $H$ does not contain $C_4$ as a subgraph. Then $\ex(n,H)=o(n^{2-1/r})$.
\end{theorem}

The proof of this result relies on ideas of Janzer~\cite{Ja18}, who found a simpler proof of Theorem~\ref{clsub} with much improved bounds. Since $K_3' = C_6$ and $\ex(n, C_6) = \Theta(n^{4/3})$, this result is tight up to the implied constant for $t = 3$ and it is plausible that it is also tight for all other~$t$.

\begin{theorem}[Janzer]
	For any integer $t\geq 3$, $\ex(n,K'_t)=O(n^{3/2-\frac{1}{4t-6}})$.
\end{theorem}

Improving another result of Conlon and Lee~\cite{CL18}, Janzer~\cite{Ja18} also obtained the following bound for the extremal number of $K_{s,t}'$, the $1$-subdivision of $K_{s,t}$.

\begin{theorem}[Janzer] \label{subbipartite}
	For any integers $2\leq s\leq t$, $\ex(n,K'_{s,t})=O(n^{3/2-\frac{1}{4s-2}})$.
\end{theorem}

This theme was again taken up in a recent paper of Kang, Kim and Liu~\cite{KKL18}, where they made the following conjecture about the $1$-subdivision of a general bipartite graph.

\begin{conjecture}[Kang--Kim--Liu]
	Let $H$ be a bipartite graph. If $\ex(n,H)=O(n^{1+\a})$ for some $\a>0$, then $\ex(n,H')=O(n^{1+\frac{\a}{2}})$.
\end{conjecture}

In particular, as $\ex(n,K_{s,t})=O(n^{2-\frac{1}{s}})$, they conjectured that $\ex(n,K'_{s,t})=O(n^{3/2-\frac{1}{2s}})$, though they were only able to push their methods to give an alternative proof of Theorem~\ref{subbipartite}. Our next result is a proof of this latter conjecture.

\begin{theorem} \label{subbipnew}
	For any integers $2\leq s\leq t$,  $\ex(n,K'_{s,t})=O(n^{3/2-\frac{1}{2s}})$.
\end{theorem}

Moreover, this result is tight when $t$ is sufficiently large compared to $s$.

\begin{corollary} \label{subbipnewcor}
	For any integer $s\geq 2$, there exists some $t_0=t_0(s)$ such that if $t\geq t_0$, then $\ex(n,K'_{s,t})=\Theta(n^{3/2-\frac{1}{2s}})$.
\end{corollary}

We now turn to another central conjecture in extremal graph theory. Following Kang, Kim and Liu~\cite{KKL18}, we say that $r\in (1,2)$ is \emph{realisable} (by $H$) if there exists a graph $H$ such that $\ex(n,H)=\Theta(n^r)$. The rational exponents conjecture of Erd\H{o}s and Simonovits (see, for example,~\cite{Er81}) states that every rational between 1 and 2 is realisable.

\begin{conjecture}[Rational exponents conjecture] \label{rational}
	For every rational number $r\in (1,2)$, there exists a graph $H$ with $\ex(n,H)=\Theta(n^r)$.
\end{conjecture}

In a recent breakthrough, Bukh and Conlon \cite{BC17} have proved that for any rational number $r\in (1,2)$ there exists a finite family $\mathcal{H}$ of graphs such that $\ex(n,\mathcal{H})=\Theta(n^r)$, where $\ex(n,\mathcal{H})$ denotes the maximal number of edges in an $n$-vertex graph which does not contain any $H\in \mathcal{H}$ as a subgraph.

However, Conjecture \ref{rational} remains wide open. In fact, until very recently only a few realisable numbers were known, namely, $1+\frac{1}{m}$ and $2-\frac{1}{m}$ for $m\geq 2$. We have already seen how the exponents $2 - 1/m$ arise from complete bipartite graphs. The exponent $1 + \frac{1}{m}$ is realisable by the theta graph $\theta_{m,\ell}$ consisting of $\ell$ internally disjoint paths of length $m$ between two vertices, with the upper bound being due to Faudree and Simonovits~\cite{FS83} and the matching lower bound for $\ell$ sufficiently large due to Conlon~\cite{C18}.  

Just a few months ago, Jiang, Ma and Yepremyan \cite{JMY18} enlarged the class of realisable exponents by proving that $7/5$ and $2-\frac{2}{2m-1}$ for $m \geq 2$ are also realisable. Subsequently, Kang, Kim and Liu \cite{KKL18} proved that for each $a,b\in \N$ with $a<b$ and $b\equiv \pm 1 \;(\bmod\; a)$, the number $2-\frac{a}{b}$ is realisable, a result which then included all known examples of realisable exponents. Their main result was a tight upper bound on the extremal number of certain graphs from which the result just mentioned for $b\equiv -1 \;(\bmod\; a)$ follows fairly easily. We now define this family of graphs.

Consider a graph $F$ with a set $R\subsetneq V(F)$ of root vertices. The {\it $\ell$-blowup} of this rooted graph is the graph obtained by taking $\ell$ vertex-disjoint copies of $F$ and identifying the different copies of $v$ for each $v\in R$. We let $H_{s,1}(r)$ be the graph consisting of vertices $x_i$ ($1\leq i\leq r-1$), $y$, $z_j$ ($1\leq j\leq s$) and $w_{j,k}$ ($1\leq j\leq s, 1\leq k\leq r-1$) and edges $x_iy$ for all $i$, $yz_j$ for all $j$ and $z_jw_{j,k}$ for all $j,k$. Then $H_{s,t}(r)$ is the rooted $t$-blowup of $H_{s,1}(r)$, with the roots being $\{x_i: 1\leq i\leq r-1\}\cup \{w_{j,k}: 1\leq j\leq s, 1\leq k\leq r-1\}$. For a picture, we refer the reader to Figure~\ref{fig:trees}, where the root vertices are marked by rectangular boxes.
The result of Kang, Kim and Liu~\cite[Lemma 3.2]{KKL18} is now as follows.

\begin{theorem}[Kang--Kim--Liu] \label{theoremKKL}
	For any integers $s,t\geq 1$ and $r\geq 2$, $\ex(n,H_{s,t}(r))=O(n^{2-\frac{s+1}{r(s+1)-1}})$.
\end{theorem}

In Section \ref{sectionKKL}, we give a new proof of this result which is significantly shorter than the original one. Combined with results of Bukh and Conlon~\cite{BC17}, Theorem~\ref{theoremKKL} easily implies that $2-\frac{s+1}{r(s+1)-1}$ is realisable for every $s\geq 1, r\geq 2$. Therefore, following Kang, Kim and Liu, we see that $2-\frac{1}{r}$ is a limit point of the set of realisable exponents for every integer $r\geq 2$.

To go further, we define $L_{s,t}(k)$ to be the graph which is the $(k-1)$-subdivision of $K_{s,t}$ with an extra vertex joined to all vertices in the part of size $t$. Put differently, this graph is the rooted $t$-blowup of $L_{s,1}(k)$, where $L_{s,1}(k)$ has vertices $u$, $v$, $w_{i,j}$ ($1\leq i\leq k, 1\leq j\leq s$) and edges $uv$, $vw_{1,j}$ ($1\leq j\leq s$), $w_{i,j}w_{i+1,j}$ ($1\leq i\leq k-1$, $1\leq j\leq s$), with roots $u,w_{k,1},\dots,w_{k,s}$. We refer the reader to Figure~\ref{fig:trees} for an illustration, where again the root vertices are marked by rectangular boxes.
We have the following result.

\begin{theorem} \label{longermain}
	For any integers $s,t,k \geq 1$, $\ex(n,L_{s,t}(k))=O(n^{1+\frac{s}{sk+1}})$.
\end{theorem}
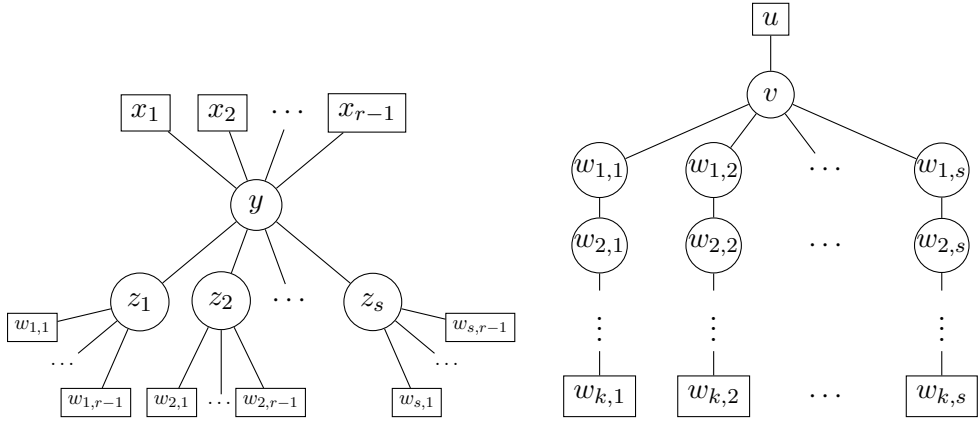
\begin{figure}
\centering
\begin{tikzpicture}[level 2/.style={sibling distance=6.5mm}]
\node [circle,draw,minimum size] (y){$y$}
child [grow=140,level distance=19mm]{node [rectangle, draw, minimum size](x1){$x_1$}}
child [grow=110, level distance=13mm]{node [rectangle, draw, minimum size](x2){$x_2$}}
child [grow=70,level distance=13mm]{node (td){$\cdots$}}
child [grow=40, level distance=19mm]{node [rectangle, draw, minimum size](xr){$x_{r-1}$}}
child [grow=220,level distance=20mm]{node [circle,draw,minimum size] (z1) {$z_{1}$} [level distance=13mm]
   	child {node [rectangle,draw,scale=.7] (w11) {$w_{1,1}$}}
   	child {node [scale=.7](cd) {$\cdots$}}
   	child {node [rectangle,draw,scale=.7] (w1r) {$w_{1,r-1}$}}
  }
child [grow=250,level distance=13.5mm]{node [circle,draw,minimum size] (z2) {$z_{2}$}[grow=down]
   	child {node [rectangle,draw,scale=.7] (w21) {$w_{2,1}$}}
   	child {node  [scale=.7](cd) {$\dots$}}
   	child {node [rectangle,draw,scale=.7] (w2r) {$w_{2,r-1}$}}
  }
  child [grow=290,level distance=13.5mm]{node (zcd) {$\cdots$}
}
  child [grow=320,level distance=20mm]{node [circle,draw,minimum size] (zs) {$z_{s}$}[level distance=13mm]
   	child {node [rectangle,draw,scale=.7] (ws1) {$w_{s,1}$}}
   	child {node  [scale=.7](cd) {$\cdots$}}
   	child {node [rectangle,draw,scale=.7] (wsr) {$w_{s,r-1}$}}
  }
;
\end{tikzpicture}\hspace{5mm}
\begin{tikzpicture}[level/.style={sibling distance=30mm/#1}, level distance=10mm]
\node [rectangle,draw,minimum size] (u){$u$}
child {node [circle,draw] (v) {$v$}
  child {node [circle,draw,inner sep=0pt] (w11) {$w_{1,1}$}
    child {node [circle,draw,inner sep=0pt] (w21) {$w_{2,1}$}
      child {node {$\vdots$}
        	child {node [rectangle,draw] (wk1) {$w_{k,1}$}}
      } 
    }
  }
  child {node [circle,draw,inner sep=0pt] (w12) {$w_{1,2}$}
    child {node [circle,draw,inner sep=0pt] (w22) {$w_{2,2}$}
      child {node {$\vdots$}
        child {node [rectangle,draw] (wk2) {$w_{k,2}$}}
      } 
    }
  }
  child {node (wcd){$\cdots$}
  }
  child {node [circle,draw,inner sep=0pt] (w1s) {$w_{1,s}$}
    child {node [circle,draw,inner sep=0pt] (w2s) {$w_{2,s}$}
      child {node {$\vdots$}
        child {node [rectangle,draw] (wks) {$w_{k,s}$}}
      } 
    }
  }
}
;
\path (w22) -- (w2s) node [midway] {$\cdots$};
\path (wk2) -- (wks) node [midway] {$\cdots$};
\end{tikzpicture}
\caption{$H_{s,1}(r)$ and $L_{s,1}(k)$} \label{fig:trees}
\end{figure}
This result has several interesting corollaries. The first is a complete resolution of Problem 5.2 from~\cite{KKL18}.

\begin{corollary} \label{longersubcor}
	For any integers $s, k \geq 1$, there exists some $t_0=t_0(s,k)$ such that if $t\geq t_0$, then $\ex(n,L_{s,t}(k))=\Theta(n^{1+\frac{s}{sk+1}})$. In particular, the exponent $1+1/k$ is a limit point of the set of realisable numbers.
\end{corollary}

Recall that $H^k$ denotes the $k$-subdivision of the graph $H$. Building on work of Kostochka and Pyber \cite{KP88} and Jiang \cite{J11}, Jiang and Seiver~\cite{JS12} gave an upper bound for the extremal number of the $k$-subdivision of a graph.

\begin{theorem}[Jiang--Seiver] \label{theoremJS}
	Let $k \geq 2$ be an even integer and let $H$ be a graph. Then $\ex(n,H^{k-1})=O(n^{1+16/k})$.
\end{theorem}

Conlon and Lee~\cite{CL18} have conjectured that the following strengthening should hold.

\begin{conjecture}[Conlon--Lee]
	Let $k \geq 2$ be an even integer and let $H$ be a graph. Then there exists some $\d>0$ such that $\ex(n,H^{k-1})=O(n^{1+1/k-\d})$.
\end{conjecture}

Our Theorem \ref{longermain} establishes this conjecture for bipartite $H$.

\begin{theorem} \label{bipartitesubdivisions}
	For any integers $s,t,k \geq 1$, $\ex(n,K_{s,t}^{k-1})=O(n^{1+\frac{s}{sk+1}})$. In particular, for any bipartite graph $H$, there exists $\delta > 0$ such that $\ex(n,H^{k-1})=O(n^{1+1/k-\d})$.
\end{theorem}

\begin{proof2}
	$K_{s,t}^{k-1}$ is a subgraph of $L_{s,t}(k)$.
\end{proof2}

\medskip

This is nearly tight, as the next proposition shows.

\begin{proposition} \label{proplower}
	For any integers $s,k \geq 1$, there exists some $t_0=t_0(s,k)$ such that if $t\geq t_0$, then $\ex(n,K_{s,t}^{k-1})=\Omega(n^{1+\frac{s-1}{sk}})$.
\end{proposition}

\smallskip

Even for subdivisions of general graphs, we obtain a large improvement on Theorem \ref{theoremJS}.

\begin{theorem}
	Let $k\geq 2$ be an even integer and let $H$ be a graph. Then there exists some $\d>0$ such that $\ex(n,H^{k-1})=O(n^{1+2/k-\d})$.
\end{theorem}

\begin{proof2}
	We have $H^{k-1}=(H^1)^{k/2-1}$. But $H^1$ is bipartite, so Theorem \ref{bipartitesubdivisions} applies.
\end{proof2}

\medskip



The rest of the paper is organised as follows. In Section \ref{sectionprelimturan}, we present some preliminary lemmas that will be used in the proofs. Then, in Section \ref{sectionmaxdegr}, we prove Theorem \ref{maxdegr}. We prove Theorem \ref{subbipnew} and Corollary \ref{subbipnewcor} in Section \ref{sectionsubbipnew}. In Section \ref{sectionKKL}, we give our new proof of Theorem~\ref{theoremKKL}, while Section \ref{sectionlongersub} contains the proofs of Theorem \ref{longermain}, Corollary \ref{longersubcor} and Proposition~\ref{proplower}. We conclude with some further remarks and questions.

\section{Preliminaries} \label{sectionprelimturan}

A common feature of our proofs is that we first assume our host graph is sufficiently regular. Let us say that a graph $G$ is $\emph{K-almost-regular}$ if $\max_{v\in V(G)}\deg(v)\leq K\min_{v\in V(G)}\deg(v)$. The reason why we may assume that our graph is almost regular is the following result of Jiang and Seiver~\cite{JS12}, which is a slight modification of a much earlier result of Erd\H os and Simonovits \cite{ES70}.

\begin{lemma}[Jiang--Seiver] \label{lemmaJS}
	Let $\e,c$ be positive reals, where $\e<1$ and $c\geq 1$. Let $n$ be a positive integer that is sufficiently large as a function of $\e$. Let $G$ be a graph on $n$ vertices with $e(G)\geq cn^{1+\e}$. Then $G$ contains a $K$-almost regular subgraph $G_{\mathrm{reg}}$ on $m\geq n^{\frac{\e-\e^2}{2+2\e}}$ vertices such that $e(G_{\mathrm{reg}})\geq \frac{2c}{5}m^{1+\e}$ and $K=20\cdot 2^{\frac{1}{\e^2}+1}$.
\end{lemma}

In Section \ref{sectionmaxdegr}, we will need a version of this lemma where $c$ can be smaller than 1.

\begin{lemma}
	Let $\e,c$ be positive reals, where $\e<1$. Let $n$ be a positive integer that is sufficiently large as a function of $\e$ and $c$. Let $G$ be a graph on $n$ vertices with $e(G)\geq cn^{1+\e}$. Then $G$ contains a $K$-almost regular subgraph $G_{\mathrm{reg}}$ on $m\geq n^{\frac{\e-\e^2}{4+4\e}}$ vertices such that $e(G_{\mathrm{reg}})\geq \frac{2c}{5}m^{1+\e}$ and $K=20\cdot 2^{\frac{1}{\e^2}+1}$.
\end{lemma}

The proof of this is the same as the proof of Lemma \ref{lemmaJS} with one straightforward modification. Nevertheless, we include it here for completeness. Note that here and throughout logarithms will be understood to be base two.

\begin{proof2}
	For convenience, we will drop all floor and ceiling signs, noting that doing so does not affect the analysis in an essential way. Let $\e,c$ be positive reals, where $\e<1$. Let $n$ be a positive integer sufficiently large as a function of $\e$ and $c$. Let $G$ be a graph on $n$ vertices with
	$e(G) \geq cn^{1+\e}$. Set $p = 2^{\frac{1}{\e^2}+1}$. We partition $V(G)$ into $2p$ almost equal parts
	$B_1,\dots,B_{2p}$, where $B_1$ consists of $\frac{n}{2p}$ vertices of the highest degrees in $G$.
	
	Suppose first that at most $\frac{c}{2}n^{1+\e}$ edges of $G$ are incident to $B_1$. We say that $G$ is of type 1. Let $H = G \setminus B_1$. Then $e(H) \geq \frac{c}{2}n^{1+\e}$. Successively remove vertices of degree less than $\frac{c}{10}n^{\e}$ from $H$ until we get stuck. Denote the remaining subgraph by $G_{\mathrm{reg}}$. Let $m=|V(G_{\mathrm{reg}})|$. Since at most $\frac{c}{10}n^{\e}\cdot n=\frac{c}{10}n^{1+\e}$ edges were removed in the deletion process, we have
	$e(G_{\mathrm{reg}}) \geq \frac{4c}{10}n^{1+\e} \geq \frac{2c}{5}m^{1+\e}$. Moreover, $\d(G_{\mathrm{reg}})\geq \frac{c}{10}n^{\e}$ by construction. Note now that $d_G(x) \geq \Delta(G_{\mathrm{reg}})$ for all $x \in B_1$ and also $\sum_{x\in B_1} d_G(x) \leq cn^{1+\e}$, since at most $\frac{c}{2}n^{1+\e}$ edges of $G$ are incident to $B_1$. Therefore, $\Delta(G_{\mathrm{reg}})(n/2p) \leq \sum_{x\in B_1} d_G(x) \leq
	cn^{1+\e}$, from which we get $\Delta(G_{\mathrm{reg}}) \leq 2pcn^{\e}$. Thus, $\Delta(G_{\mathrm{reg}})/\d(G_{\mathrm{reg}}) \leq 2pcn^{\e}/ \frac{c}{10}n^{\e} = 20p$.
	So $G_{\mathrm{reg}}$ is $K$-almost-regular. Since  
	$$m \geq 2e(G_{\mathrm{reg}})/\Delta(G_{\mathrm{reg}}) \geq \frac{4c}{5}n^{1+\e}/2pcn^{\e} = \frac{2}{5p}n \geq n^{\frac{\e-\e^2}{4+4\e}}$$ 
	for large $n$, the lemma holds in this case.
	
	Suppose now that more than $\frac{c}{2}n^{1+\e}$ edges of $G$ are incident to $B_1$. We say that
	$G$ is of type 2. By averaging, for some $j \in \{2,\dots, 2p\}$, the subgraph $G_1$
	of $G$ induced by $B_1\cup B_j$ has more than $\frac{1}{2p}\frac{c}{2}n^{1+\e}=\frac{c}{4p}n^{1+\e}$ edges. Let $n_1 = |V(G_1)|$. Then $n_1 = n/p$. Note that $cn_1^{1+\e} = c(\frac{n}{p})^{1+\e}=\frac{c}{p}n^{1+\e}\frac{1}{p^{\e}}\leq \frac{c}{4p}n^{1+\e}$, using that $p^{\e}= 2^{(\frac{1}{\e^2}+1)\e}\geq 4$. So $e(G_1) \geq cn_1^{1+\e}$.
	
	We can now replace $G$ with $G_1$ and repeat the analysis. If $G_1$ is of type 1, we terminate. If $G_1$ is of type 2, we define $G_2$ from $G_1$ in the same way we defined $G_1$ from $G$.	We continue like this as long as the new graph $G_i$ is of type 2. We terminate when
	$G_i$ is of type 1 for the first time. With $G_0 = G$, let $k$ be the smallest $i$ such that
	$G_i$ is of type 1. Then $|V(G_k)| = \frac{n}{p^k}$ and $e(G_k) \geq \frac{c}{(4p)^k}n^{1+\e}$. Since $e(G_k) \leq |V(G_k)|^2$, we have $\frac{c}{(4p)^k}n^{1+\e}\leq \frac{n^2}{p^{2k}}$. Thus, $(\frac{p}{4})^k\leq \frac{n^{1-\e}}{c}\leq n^{1-\e+\frac{\e(1-\e)^2}{2(1+\e^2)}}$ as $n$ is sufficiently large, so $k\leq \left(1-\e+\frac{\e(1-\e)^2}{2(1+\e^2)}\right)\frac{\log n}{\log (p/4)}$. Since $n_k :=|V(G_k)| = n/p^k$, $\log n_k = \log n-k\log p\geq \left(1-\left(1-\e+\frac{\e(1-\e)^2}{2(1+\e^2)}\right)\frac{\log p}{\log(p/4)}\right)\log n$. Plugging in $p=2^{\frac{1}{\e^2}+1}$, we get 
	$$\log n_k\geq \left(1-\left(1-\e+\frac{\e(1-\e)^2}{2(1+\e^2)}\right)\frac{\frac{1}{\e^2}+1}{\frac{1}{\e^2}-1}\right)\log n=\frac{\e-\e^2}{2+2\e}\log n$$ 
	and, therefore, $n_k\geq n^{\frac{\e-\e^2}{2+2\e}}.$
	Since $G_k$ is of type 1, our earlier arguments imply that it contains a subgraph $G_{\mathrm{reg}}$ on $m$ vertices, where $m\geq \frac{2}{5p}n_k\geq n^{\frac{\e-\e^2}{4+4\e}}$ for large $n$. Furthermore, $e(G_{\mathrm{reg}})\geq \frac{2c}{5}m^{1+\e}$ and $G_{\mathrm{reg}}$ is $K$-almost-regular, as required.
\end{proof2}

\medskip

We will in fact need a version of this lemma which gives an almost-regular bipartite subgraph. We say that a bipartite graph $G$ with bipartition $A\cup B$ is \emph{balanced} if $\frac{1}{2}|B|\leq |A|\leq 2|B|$. The proof of the following lemma is almost identical to the proof of Lemma 2.3 in \cite{CL18} and is therefore omitted.

\begin{lemma} \label{lemmaJSmodified}
	Let $\e,c$ be positive reals, where $\e<1$. Let $n$ be a positive integer that is sufficiently large as a function of $\e$ and $c$. Let $G$ be a graph on $n$ vertices with $e(G)\geq cn^{1+\e}$. Then $G$ contains a $K$-almost regular balanced bipartite subgraph $G_{\mathrm{bip}}$ on $m\geq n^{\frac{\e-\e^2}{4+4\e}}$ vertices such that $e(G_{\mathrm{bip}})\geq \frac{c}{10}m^{1+\e}$ and $K=60\cdot 2^{\frac{1}{\e^2}+1}$.
\end{lemma}

The main focus of this paper is on proving upper bounds for extremal numbers. However, in many cases we can use a result of Bukh and Conlon~\cite{BC17} to show that there is a matching lower bound. To say more, suppose that $F$ is a graph with a set of roots $R\subsetneq V(F)$. For any non-empty $S\s V(F)\setminus R$, let $e_S$ be the number of edges in $F$ adjacent to $S$. Set $\rho_F(S)=\frac{e_S}{|S|}$ and $\rho(F)=\rho_F(V(F)\setminus R)$. We say that $(F,R)$ (or $F$ if $R$ is clear) is \emph{balanced} if $\rho(F)\leq \rho_F(S)$ holds for every non-empty $S\s V(F)\setminus R$. Let us write $\ell\ast F$ for the $\ell$-blowup of the rooted graph $F$, as defined in Section \ref{sectionintroturan}. The result of Bukh and Conlon is now as follows.

\begin{lemma}[Bukh--Conlon] \label{lemmaBC}
	Let $F$ be a balanced bipartite rooted graph with $\rho(F)>0$. Then there is some $\ell_0\in \N$ such that, for every $\ell\geq \ell_0$, $\ex(n,\ell\ast F)=\Omega(n^{2-\frac{1}{\rho(F)}})$.
\end{lemma}

The notation we use in the remaining sections is mostly standard. For a graph $G$ and $v\in V(G)$, we write $N_G(v)$ (or $N(v)$ if $G$ is clear) for the neighbourhood of $v$ in $G$. We also write $d_G(v)$ or $d(v)$ for the degree of $v$. Finally, if $u_1,\dots,u_r\in V(G)$, then we write $d_G(u_1,\dots,u_r)=d(u_1,\dots,u_r)=|N_G(u_1)\cap \dots \cap N_G(u_r)|$.

\section{$C_4$-free bipartite graphs with max degree $r$ on one side} \label{sectionmaxdegr}

In this section, we prove Theorem \ref{maxdegr}. In order to prove this theorem, we may clearly assume that all the degrees in one part of $H$ are \emph{exactly} $r$. Then Lemma \ref{lemmaJSmodified} reduces Theorem \ref{maxdegr} to the following statement.

\begin{theorem} \label{mainreduced}
	Let $r\geq 2$ be an integer, let $K\geq 1$ be fixed and let $H$ be a bipartite graph such that in one of the parts all the degrees are exactly $r$ and $H$ does not contain $C_4$ as a subgraph. Then, for any constant $c > 0$, there exists $n_0$ such that if $n \geq n_0$ and $G$ is a $K$-almost-regular balanced bipartite graph with bipartition $A\cup B$, $|B|=n$, and minimum degree $\d\geq c n^{1-1/r}$, then $G$ contains a copy of $H$.
\end{theorem}

We will need the following generalisation of a simple lemma from~\cite{CL18}.

\begin{lemma} \label{locallydenser}
	Let $r\geq 2$ be an integer and let $G$ be a bipartite graph with bipartition $A\cup B$, $|B|=n$, and minimum degree at least $\d$ on the vertices in $A$. Then, for any subset $U\s A$ with $|U|\geq \frac{rn}{\d}$,
	$$ \sum_{u_1\dots u_r\in {U \choose r}} d(u_1,\dots,u_r)\geq \frac{\d^r}{r^rn^{r-1}}|U|^r\geq \frac{\d^r}{r^rn^{r-1}}{|U| \choose r}.$$
\end{lemma}

\begin{proof2}
	Writing $d_U(v)$ for $|N_G(v)\cap U|$, we have that
	\begin{align*}
		\sum_{u_1\dots u_r\in {U \choose r}} d(u_1,\dots,u_r)&=\sum_{b\in B} {d_U(b) \choose r}\geq n{\sum_{b\in B} d_U(b)/n \choose r} \\
		&=n{\sum_{u\in U} d(u)/n \choose r}\geq n{\d |U|/n \choose r} \\
		&\geq n\Big(\frac{\d |U|}{rn}\Big)^r=\frac{\d^r}{r^rn^{r-1}}|U|^r,
	\end{align*}
	where the first inequality follows from the convexity of ${x \choose r}$ and in the last inequality we used that $|U|\geq \frac{rn}{\d}$.
\end{proof2}

\medskip

Given a bipartite graph $G$ with bipartition $A\cup B$, the \emph{neighbourhood r-graph} is the weighted $r$-uniform hypergraph $W_G$ on vertex set $A$ where the weight of the hyperedge $u_1\dots u_r$ (for $u_1,\dots,u_r$ distinct) is $d(u_1,\dots,u_r)$. For a subset $U\s A$, we write $W(U)$ for the total weight in $U$, i.e., $W(U)=\sum_{u_1\dots u_r\in {U \choose r}} d(u_1,\dots,u_r)$. In this language, the conclusion of Lemma \ref{locallydenser} is that $W(U)\geq  \frac{\d^r}{r^rn^{r-1}}{|U| \choose r}$.

In the next definition, for a weighted $r$-graph $W$ on vertex set $A$ and $u_1,\dots,u_r\in A$, we write $W(u_1,\dots,u_r)$ for the weight of the hyperedge $u_1\dots u_r$. Moreover, in what follows we fix $r\geq 2$ and a bipartite graph $H$ with the property that in one part all the degrees are exactly $r$. Let $h=|V(H)|$.

\begin{definition}
	Let $W$ be a weighted $r$-graph on vertex set $A$ and let $u_1,\dots,u_r\in A$ be distinct. We say that $u_1\dots u_r$ is a \emph{light edge} if $1\leq W(u_1,\dots,u_r)<{h\choose r}$ and that it is a \emph{heavy edge} if $W(u_1,\dots,u_r)\geq {h\choose r}$.
\end{definition}

Note that if there is a $K^{(r)}_{h}$ in $W_G$ formed by heavy edges, then clearly there is a copy of $H$ in $G$. This observation is an important ingredient in our next lemma.

\begin{lemma} \label{manylightr}
	Let $G$ be an $H$-free bipartite graph with bipartition $A\cup B$, $|B|=n$, and suppose that $W(A)\geq 2h^rn$. Then the number of light edges in $W_G$ is at least $\frac{W(A)}{2h^{2r}}$.
\end{lemma}

\begin{proof2}
	Suppose $B=\{b_1,\dots,b_n\}$. Let $k_i=|N_G(b_i)|$ and suppose that $k_i\geq h$ for some $i$. As $G$ is $H$-free, there is no $K^{(r)}_{h}$ in $W\lbrack N_G(b_i)\rbrack$ formed by heavy edges. Since $\ex(t,K^{(r)}_{h})\leq (1-1/{h \choose r}){t \choose r}$ holds for $t\geq h$, the number of light edges in $W[N_G(b_i)]$ is at least $\frac{{k_i \choose r}}{{h \choose r}}$. But
	
	\begin{equation*}
		\sum_{i: k_i<h} {k_i \choose r}<h^rn\leq \frac{W(A)}{2},
	\end{equation*} so
	
	\begin{equation*}
		\sum_{i: k_i\geq h} {k_i \choose r}\geq \frac{W(A)}{2}.
	\end{equation*}
	Since every light edge is present in at most ${h \choose r}$ of the sets $N_G(b_i)$, it follows that the total number of light edges is at least
	
	$$ \frac{1}{{h \choose r}}\sum_{i: k_i\geq h} \frac{{k_i \choose r}}{{h \choose r}}\geq \frac{W(A)}{2h^{2r}},$$
	as required.
\end{proof2}

\begin{corollary} \label{lightcorollaryr}
	Let $G$ be an $H$-free bipartite graph with bipartition $A\cup B$, $|B|=n$, and minimum degree at least $\d$ on the vertices in $A$. Then, for any subset $U\subset A$ with $|U|\geq \frac{2hrn}{\d}$, the number of light edges in $W_G\lbrack U\rbrack$ is at least $\frac{\d^r}{2h^{2r}r^rn^{r-1}}{|U| \choose r}$.	
\end{corollary}

\begin{proof2}
	By Lemma \ref{locallydenser}, we have $W(U)\geq \frac{\d^r}{r^rn^{r-1}}|U|^r\geq 2^rh^rn$. Hence, the result follows by applying Lemma \ref{manylightr} to the graph $G\lbrack U\cup B\rbrack$. 
\end{proof2}

\medskip

We now recall Definition 5 from \cite{KNRS10}.

\begin{definition}
	An $r$-uniform hypergraph $\mathcal{G}=(V,E)$ is $(\rho,d)$-dense if, for any subset $U\s V$ of size $|U|\geq \rho |V|$, $e_{\mathcal{G}}(U)\geq d{|U| \choose r}$.
\end{definition}

Recall also that a \emph{linear hypergraph} is a hypergraph where any two edges intersect in at most one vertex. The following result follows from Theorem 7 in \cite{KNRS10}.

\begin{theorem}[Kohayakawa--Nagle--R\"odl--Schacht] \label{embed}
	Let $\L$ be a linear $r$-uniform hypergraph on $\ell$ vertices. Then, for every $d>0$, there exist $\rho=\rho(\L,d)>0$, $\e=\e(\L,d)>0$ and $n_0=n_0(\L,d)$ such that every $(\rho,d)$-dense $r$-uniform hypergraph $\G=(V,E)$ on $n\geq n_0$ vertices contains at least $\e|V|^\ell$ copies of $\L$.
\end{theorem}

\smallskip

We are now in a position to complete the proof of Theorem \ref{mainreduced}.

\begin{proof}[\textnormal{\textbf{Proof of Theorem \ref{mainreduced}}}]
	We may assume that $\d\leq n^{1-1/2r}$, as we already know that $\ex(n,H)=O(n^{2-1/r})$. Suppose that $G$ is $H$-free. Define $\G$ to be the $r$-uniform (simple) hypergraph whose vertex set is $A$ and whose edges are precisely the light edges of $W_G$. By Corollary \ref{lightcorollaryr}, for any $U\s A$ with $|U|\geq \frac{2hrn}{\d}$, we have $$e_{\G}(U)\geq \frac{\d^r}{2h^{2r}r^rn^{r-1}}{|U| \choose r}\geq \frac{c^r}{2h^{2r}r^r}{|U| \choose r}.$$
	Suppose $H$ has bipartition $X\cup Y$ with every vertex in $Y$ having degree $r$. Define $\L$ to be the $r$-uniform hypergraph whose vertex set is $X$ and whose edges are the neighbourhoods $N_H(y)$ for $y\in Y$. Since $H$ does not contain a $C_4$, it follows that $\L$ is linear. Let $d=\frac{c^r}{2h^{2r}r^r}$ and choose $\rho>0$, $\e>0$ and $n_0$ as in the conclusion of Theorem \ref{embed}. Note that for $n$ sufficiently large, we have $\frac{2hrn}{\d}< \rho|A|$, so $\G$ is $(\rho,d)$-dense and consequently contains at least $\e |A|^{|X|}$ copies of $\L$. 
	All these copies of $\L$ provide homomorphic copies of $H$ in $G$. To see this, suppose that we have a copy of $\L$ in $\G$. We map the vertices of $X$ to the vertices of $A$ so that the copy of $\L$ in $X$ maps isomorphically onto the copy of $\L$ in $\G$. Call this map $f$. To complete the embedding, for each $y \in Y$, we map $y$ to a vertex in the neighbourhood of $f(N_H(y))$. Note that this neighbourhood is non-empty because $N_H(y)$ is an edge of $\L$ and each such edge was mapped under $f$ to an edge of $\G$, which, by definition, has a non-empty neighbourhood. However, some of the resulting copies of $H$ may be degenerate in the sense that distinct vertices in $Y$ may be mapped to the same vertex in~$B$.
	
	We now give an upper bound for the number of degenerate copies of $H$, counting only those copies that were obtained by the method above. Any such degenerate copy must contain some $u\in B$ and $v_1,\dots,v_{r+1}\in N_G(u)$ with $v_1\dots v_r$ a light edge in $W_G$. The number of possible choices for such a configuration is at most $(2n)^r\cdot {h \choose r}\cdot K\d$, since we can choose $v_1,\dots,v_r$ in at most $(2n)^r$ ways (since $|A|\leq 2n$), then we can choose $u$ in at most ${h \choose r}$ ways (since $v_1\dots v_r$ is a light edge) and, finally, we can choose $v_{r+1}$ in at most $K\d$ ways (since $\Delta(G)\leq K\d$). 
	But the number of ways to extend this to a copy of $H$ is at most $(2n)^{|X|-r-1}\cdot {h \choose r}^{{|X| \choose r}}$, because we can map those vertices in $X$ that have not been mapped in at most $(2n)^{|X|-r-1}$ ways and, given any choice for the images of $X$, there are at most ${h \choose r}$ possible choices for the image of each $y\in Y$, since we are only counting those copies of $H$ in which $N_H(y)$ is mapped to a light edge. Thus, of the $\e|A|^{|X|}$ copies of $H$ that we found, at most ${h \choose r}^{{|X| \choose r}+1}K\d (2n)^{|X|-1}$ are degenerate. Since $\d \leq n^{1-1/2r}$ and $|A|\geq n/2$, for sufficiently large $n$ we obtain a non-degenerate copy of $H$.
\end{proof}

\section{The 1-subdivision of $K_{s,t}$} \label{sectionsubbipnew}

In this section, we prove Theorem \ref{subbipnew} and Corollary \ref{subbipnewcor}. By Lemma \ref{lemmaJSmodified}, Theorem \ref{subbipnew} reduces to the following. 

\begin{theorem} \label{bipreduced}
	Let $2\leq s\leq t$ be fixed integers and let $K\geq 1$ be a constant. Suppose that $G$ is a balanced bipartite graph with bipartition $A\cup B$, $|B|=n$, such that $G$ is $K$-almost-regular with minimum degree $\d=\omega(n^{1/2-\frac{1}{2s}})$. Then, for $n$ sufficiently large, $G$ contains a copy of~$K_{s,t}'$.
\end{theorem}

Note that the assumption that $\d= \omega(n^{1/2-\frac{1}{2s}})$ is purely for notational convenience. The proof goes through in exactly the same way when we replace this assumption with $\d \geq C n^{1/2-\frac{1}{2s}}$ for a sufficiently large constant $C$, but using $\omega$ allows us to ignore how this constant changes at each step. We will use a similar convention in the following sections.

In what follows, let $2\leq s\leq t$ be fixed integers and $K\geq 1$ a constant. Given a bipartite graph $G$ with bipartition $A\cup B$, we write $W_G$ for the neighbourhood graph of $G$ on vertex set $A$. Recall from Section \ref{sectionmaxdegr} that this is the weighted graph where the weight $W(u,v)$ of the pair $uv$ is $d_G(u,v)$. For distinct $u,v\in A$, we say that $uv$ is a \emph{light edge} (in $W_G$) if $1\leq W(u,v)< {s+t \choose 2}$ and a \emph{heavy edge} if $W(u,v)\geq {s+t \choose 2}$.

Let us first give a rough sketch of the proof of Theorem \ref{bipreduced}. Note that any $K_{s,t}$ in the neighbourhood graph $W_G$ yields a homomorphic copy of $K_{s,t}'$ in $G$. However, it may be a degenerate copy. Nevertheless, the first step is to find many copies of $K_{s,t}$ in~$W_G$. By the degree conditions, the total weight in $W_G$ is $\omega(n^{2-\frac{1}{s}})$, so if $W_G$ was a simple graph rather than a weighted graph, we could find $\omega(n^{s})$ copies of $K_{s,t}$. Thus, we first prove that there are $\omega(n^{2-\frac{1}{s}})$ pairs in $A$ which determine an edge (of arbitrary positive weight) in~$W_G$.

Once we have established this, we run the usual proof for finding a $K_{s,t}$, namely, we double count the number of $s$-stars in the graph $W_G$ (or, more precisely, in the simple graph obtained by replacing each edge of $W_G$ by a simple edge). On average, a set of size $s$ will have a common neighbourhood of size $\omega(1)$. This provides us with $\omega(n^s)$ copies of $K_{s,t}$ in $W_G$. 
We then argue that if all of these yield degenerate copies of $K_{s,t}'$ in $G$, then some degenerate copies can be patched together to find an $s$-set $S$ with abnormally large common neighbourhood. 
More precisely, there is an $s$-set $S\s N_G(b)$ for some $b \in B$ with common neighbourhood of size~$\omega(\d)$ in $W_G$, which is very large compared to the typical size~$\omega(1)$ of the common neighbourhood of an $s$-set.
It is then fairly easy to use this property to show that there must be a $K_{s,t}$ in $W_G$ (with the part of order $s$ being equal to $S$) that gives a non-degenerate copy of $K_{s,t}'$ in $G$.

\medskip

For distinct vertices $u_1,\dots,u_s\in A$, we write $N'_W(u_1,\dots,u_s)$ for the set of those $x\in A$ which are distinct from all $u_i$ and for which there exist distinct $b_1,\dots,b_s\in B$ such that $b_i\in N_G(u_i)\cap N_G(x)$ for all $i$. Informally, the $xu_i$ are edges in $W_G$ coming from distinct elements of $B$. We also write $d'_W(u_1,\dots,u_s)=|N'_W(u_1,\dots,u_s)|$. 

Roughly speaking, the next lemma gives a lower bound on the number of $s$-stars in the graph $W_G$, as promised in the sketch above.

\begin{lemma} \label{doublecount}
	Let $G$ be a $K_{s,t}'$-free balanced bipartite graph with bipartition $A\cup B$, $|B|=n$, such that $G$ is $K$-almost-regular with minimum degree $\d=\omega(n^{1/2-\frac{1}{2s}})$. Then $$\sum  d_W'(u_1,\cdots,u_s)=\Omega(n\d^{2s}),$$ where the sum is taken over all choices of distinct $u_1,\dots,u_s\in A$.
\end{lemma}

In the proof of this lemma, we make use of the following result, which is an easy consequence of Lemma 10 from \cite{Ja18}.

\begin{lemma}
	Let $G$ be a $K_{s,t}'$-free bipartite graph with bipartition $A\cup B$, $|B|=n$, and suppose that $W(A)\geq 8(s+t+1)^2n$. Then the number of light edges in $W_G$ is at least $\frac{W(A)}{4(s+t+1)^3}$.
\end{lemma}

Since the total weight of edges in $W_G$ is at least $n \binom{\d}{2}$, we have the following corollary.

\begin{corollary} \label{light}
	Let $G$ be a $K_{s,t}'$-free bipartite graph with bipartition $A\cup B$, $|B|=n$, such that $G$ is $K$-almost-regular with minimum degree $\d=\omega(n^{1/2-\frac{1}{2s}})$. Then the number of light edges in $W_G$ is at least $\frac{n{\d \choose 2}}{4(s+t+1)^3}$.
\end{corollary}

\begin{proof}[\textnormal{\textbf{Proof of Lemma \ref{doublecount}}}]
	The proof proceeds by double counting the number of $(s+1)$-tuples $(x,u_1,\dots,u_s)\in A^{s+1}$ with the following properties:
	
	\begin{enumerate}[label=(\roman*)]
		\item Each $xu_i$ is a light edge in $W_G$.
		
		\item For any $i\neq j$, we have $N_G(x)\cap N_G(u_i)\cap N_G(u_j)=\emptyset$.
	\end{enumerate}
	 In particular, $u_i$ and $u_j$, $i\neq j$, are distinct, since otherwise (i) and (ii) contradict one another. 
    Note also that if these properties are satisfied, then $x\in N'_W(u_1,\dots,u_s)$. In fact, the same conclusion holds even if (i) does not require the edges to be light.

	For any $x\in A$, let $d_{1}(x)$ be the number of light edges adjacent to $x$ in $W_G$. Then, by Corollary~\ref{light}, we have $\sum_{x\in A} d_{1}(x)= \Omega(n\d^2)$. The number of $(s+1)$-tuples $(x,u_1,\dots,u_s)$ satisfying (i) is $\sum_{x\in A} d_{1}(x)^s\geq |A|\left(\frac{\sum_{x\in A} d_{1}(x)}{|A|}\right)^s=\Omega(n\d^{2s})$. But, of all these $(s+1)$-tuples, there are at most $s^2\cdot 2n\cdot (K\d)\cdot (K\d)^2\cdot (K^2\d^2)^{s-2}$ that do not satisfy (ii). This is because, for a fixed $i,j$, at most $2n\cdot (K\d)\cdot (K\d)^2\cdot (K^2\d^2)^{s-2}$ choices violate (ii), since there are at most $2n$ ways to choose $x$, then at most $K\d$ ways to choose an element to be in $N_G(x)\cap N_G(u_i)\cap N_G(u_j)$ and, given any such choice, there are at most $(K\d)^2$ choices for $u_i$ and $u_j$. Finally, there are at most $(K\d)^2$ choices for every other $u_k$ since the degree of $x$ in $W_G$ is at most $(K\d)^2$. Therefore, the total number of $(s+1)$-tuples satisfying (i) but not (ii) is $O(n\d^{2s-1})$, completing the proof.
\end{proof}

We now derive some consequences of the graph being $K'_{s,t}$-free.

\begin{lemma} \label{onecolour}
	Let $G$ be a $K_{s,t}'$-free bipartite graph with bipartition $A\cup B$. Suppose that for some distinct $u_1,\dots,u_s\in A$,  $d'_W(u_1,\dots,u_s)=\omega(1)$. Then there exist $b\in B, 1\leq k\leq s$ and a subset $X\s N'_W(u_1,\dots,u_s)$ consisting of at least $\frac{d'_W(u_1,\dots,u_s)}{2s^2t}$ elements such that~$X\cup \{u_k\}\subset N_G(b)$.
\end{lemma}

\begin{proof2}
	Pick a maximal subset $Y=\{y_1,\dots,y_r\}\s N'_W(u_1,\dots,u_s)$ with the property that there exist distinct $c_{ij}\in B$ with $c_{ij}\in N_G(u_i)\cap N_G(y_j)$ for all $1\leq i\leq s$, $1\leq j\leq r$. Since $G$ is $K_{s,t}'$-free, it follows that $r< t$. For any $x\in N'_W(u_1,\dots,u_s)\setminus Y$, there exist distinct $b_i\in B$ for $1\leq i\leq s$ such that $b_i\in N_G(x)\cap N_G(u_i)$. By the maximality of $Y$, there exist some $c(x) \in \{c_{ij}: 1\leq i\leq s,1\leq j\leq r\}$ and $1\leq k(x)\leq s$ such that $b_{k(x)}=c(x)$. By the pigeonhole principle, there exist $1\leq k\leq s$ and $b\in \{c_{ij}: 1\leq i\leq s,1\leq j\leq r\}$ such that for at least $\frac{|N'_W(u_1,\dots,u_s)\setminus Y|}{s^2r}$ choices of $x\in N'_W(u_1,\dots,u_s)\setminus Y$ we have $k(x)=k$ and $b_{k(x)}=b$. This choice for $b$ and $k$ satisfies the conclusion of the lemma.
\end{proof2}

\medskip

In the next result, $R(s,s+t)$ denotes the usual Ramsey number.

\begin{corollary} \label{cor16}
	Let $G$ be a $K_{s,t}'$-free bipartite graph with bipartition $A\cup B$. Suppose that for some distinct $u_1,\dots,u_s\in A$,  $d'_W(u_1,\dots,u_s)=\omega(1)$. Then there exist $b\in B, 1\leq k\leq s$ and $\Omega(d'_W(u_1,\dots,u_s)^s)$ $s$-sets $\{x_1,\dots,x_s\}\s N'_W(u_1,\dots,u_s)$ such that $\{x_1,\dots,x_s\}\cup \{u_k\}\s N_G(b)$ and no $x_ix_j$ is a heavy edge.
\end{corollary}

\begin{proof2}
	Choose $b\in B$, $1\leq k\leq s$ and $X\s N_W'(u_1,\dots,u_s)$ as in the conclusion of Lemma~\ref{onecolour}. Since $G$ is $K_{s,t}'$-free, there is no $K_{s+t}$ in $W_G$ formed by heavy edges. Thus, in each subset of size $R(s,s+t)$ in $X$, there exists an $s$-set which does not span any heavy edge. By averaging over all subsets of $X$ of size $R(s,s+t)$, it follows that the number of $s$-sets in $X$ which do not span any heavy edge is at least
	\begin{align*}
	\frac{1}{\binom{|X|-s}{R(s,s+t)-s}}\binom{|X|}{R(s,s+t)}=
	    \frac{1}{{R(s,s+t) \choose s}}\cdot {|X| \choose s}.
	\end{align*}
	Since $|X|\geq d'_W(u_1,\dots,u_s)/2s^2t$, the RHS above is $\Omega(d'_W(u_1,\dots,u_s)^s)$.
\end{proof2}

\medskip

With these results in hand, we are ready to conclude the proof of Theorem~\ref{bipreduced}.

\begin{proof}[\textnormal{\textbf{Proof of Theorem \ref{bipreduced}}}]
	Suppose that $G$ does not contain a copy of $K_{s,t}'$.
	
	\smallskip
	
	\noindent \emph{Claim.} There exist distinct $x_1,\dots,x_s\in A$ such that no $x_ix_j$ is heavy and the number of $u\in A$ with $N_G(x_i)\cap N_G(u)\neq \emptyset$ for all $i$ is $\omega(\d)$.
	
	\medskip
	
	\noindent \emph{Proof of Claim.} Since $\delta=\omega(n^{1/2-\frac{1}{2s}})$, it follows that $n\d^{2s}=\omega(n^s)$. Choose a sequence $f(n)=\omega(1)$ with $n\d^{2s}=\omega(n^sf(n))$. Then, by Lemma \ref{doublecount}, we have $\sum d'_W(u_1,\dots,u_s)=\Omega(n\d^{2s})$, where the sum is over distinct $u_1,\dots,u_s\in A$ with $d'_W(u_1,\dots,u_s)\geq f(n)$. Now, by Corollary \ref{cor16}, for each such $u_1,\dots,u_s$, there exist $b\in B, 1\leq k\leq s$ and $\Omega(d'_W(u_1,\dots,u_s)^s)$ $s$-sets $\{x_1,\dots,x_s\}\s N'_W(u_1,\dots,u_s)$ such that $\{x_1,\dots,x_s\}\cup \{u_k\}\s N_G(b)$ and no $x_ix_j$ is a heavy edge. It follows by Jensen's inequality that there are $\Omega(n^s(\frac{n\d^{2s}}{n^s})^s)$ $2s$-tuples $(x_1,\dots,x_s,u_1,\dots,u_s)\in A^{2s}$ with the following properties:
	
	\begin{enumerate}[label=(\roman*)]
		\item All $x_i$ and $u_j$ are distinct.
		
		\item There exist $b\in B$ and $ k\in\{1,\dots, s\}$ such that $\{x_1,\dots,x_s\}\cup \{u_k\}\s N_G(b)$.
		
		\item For each $i,j$, $N_G(x_i)\cap N_G(u_j)\neq \emptyset$.
		
		\item No $x_ix_j$ determines a heavy edge in $W_G$.		

	\end{enumerate}
    Note that $n^s(\frac{n\d^{2s}}{n^s})^s=\omega(n\d^{2s})$, as $s>1$. However, there are at most $s\cdot n\cdot (K\d)^{s+1}$ ways to choose $k,b,x_1,\dots,x_s,u_k$ such that property (ii) holds. Thus, for at least one such choice, there are $\omega(\frac{n\d^{2s}}{n\d^{s+1}})=\omega(\d^{s-1})$ ways to extend to a suitable $2s$-tuple. The corresponding $x_1,\dots,x_s$ then satisfy the required conclusion. \qed
	
	\medskip
	
	Now take such $x_1,\dots,x_s\in A$. Since no $x_ix_j$ is heavy, we have $|\cup_{i<j} (N_G(x_i)\cap N_G(x_j))|=O(1)$, so the number of $u\in A$ such that there are $1\leq i<j\leq s$ with $N_G(x_i)\cap N_G(x_j)\cap N_G(u)\neq \emptyset$ is $O(\d)$. Thus, by the claim, there is a set $U\s A$ of $\omega(\d)$ vertices, distinct from $x_1,\dots,x_s$, such that for each $u\in U$, there are distinct $b_1,\dots,b_s\in B$ with $b_i\in N_G(x_i)\cap N_G(u)$ for all $i$. Take a maximal subset $U'=\{u_1,\dots,u_r\}\s U$ such that there exist distinct $c_{ij}\in N_G(x_i)\cap N_G(u_j)$ for all $1\leq i\leq s$, $1\leq j\leq r$. If $r\geq t$, then there is a $K_{s,t}'$ in $G$, so we have $r<t$. For any $v\in U\setminus U'$, there exist distinct $b_i\in N_G(x_i)\cap N_G(v)$. By the maximality of $U'$, we must have $b_i=c_{jk}$ for some $i,j,k$. Therefore, $v\in \cup_{1\leq j\leq s,1\leq k\leq r} N_G(c_{jk})$. So $U\setminus U'\s \cup_{1\leq j\leq s,1\leq k\leq r} N_G(c_{jk})$. But then $|U|< t+st\cdot K\d$, which contradicts $|U|=\omega(\d)$.
\end{proof}

Given Theorem \ref{subbipnew}, it is not hard to deduce Corollary \ref{subbipnewcor}. Indeed, note that $K'_{s,t}$ is the rooted $t$-blowup of $K'_{s,1}$ with the roots being the $s$ leaves. This rooted graph is balanced and bipartite with $\rho(K'_{s,1})=\frac{2s}{s+1}$, so Lemma \ref{lemmaBC} gives that $\ex(n,K'_{s,t})=\Omega(n^{2-\frac{s+1}{2s}})$ when $t$ is sufficiently large compared to $s$. Combining this with Theorem \ref{subbipnew}, Corollary \ref{subbipnewcor} follows.

\section{A short proof of a result of Kang, Kim and Liu} \label{sectionKKL}

Recall that $H_{s,1}(r)$ is the graph consisting of vertices $x_i$ ($1\leq i\leq r-1$), $y$, $z_j$ ($1\leq j\leq s$) and $w_{j,k}$ ($1\leq j\leq s, 1\leq k\leq r-1$) and edges $x_iy$ for all $i$, $yz_j$ for all $j$ and $z_jw_{j,k}$ for all $j,k$. Moreover, $H_{s,t}(r)$ is the rooted $t$-blowup of $H_{s,1}(r)$, with the roots being $\{x_i: 1\leq i\leq r-1\}\cup \{w_{j,k}: 1\leq j\leq s, 1\leq k\leq r-1\}$.

In this section, we prove Theorem \ref{theoremKKL}. 
By Lemma \ref{lemmaJS}, it suffices to prove the following.

\begin{theorem} \label{newexponents}
	Let $s,t\geq 1$ and $r\geq 2$ be fixed integers and $K\geq 1$ a constant. Suppose that $G$ is a $K$-almost-regular graph on $n$ vertices with minimum degree $\d= \omega(n^{1-\frac{s+1}{r(s+1)-1}})$. Then, for $n$ sufficiently large, $G$ contains a copy of $H_{s,t}(r)$.
\end{theorem}

In what follows, let $s,t\geq 1$ and $r\geq 2$ be fixed integers and let $K\geq 1$ be a constant. Let $H=H_{s,t}(r)$. The constant $L$ will be chosen suitably in terms of $s$, $t$, $r$ and $K$, while $n$ will always be sufficiently large in terms of $s$, $t$, $r$, $K$ and $L$. As a shorthand, we will now write $d_G(S)$ for the size of the common neighbourhood $N_G(S)$ of a set $S$.

\begin{definition}
	An $r$-set $S\s V(G)$ is called an \emph{$r$-edge} if $d_G(S)>0$. The weight of $S$ is~$d_G(S)$. $S$ is called a \emph{light $r$-edge} if $1\leq d_G(S)\leq L$ and a \emph{heavy $r$-edge} if $d_G(S)>L$.
\end{definition}

\begin{lemma} \label{rlight}
     Let $G$ be an $H$-free $K$-almost-regular graph on $n$ vertices with minimum degree $\d=\omega(n^{1-\frac{1}{r-1}})$. Then the total weight on heavy $r$-edges is at most an $f_L$-proportion of the total weight of $r$-edges, where $f_L \rightarrow 0$ as $L \rightarrow \infty$.
\end{lemma}

\begin{proof2}
	First note that for any $r-1$ distinct vertices $x_1,\dots,x_{r-1}$, we cannot have $m=m_{s,t,r}=t+s(r-1)$ vertices in $N(x_1)\cap \dots \cap N(x_{r-1})$ such that any $r$ of them form an edge of weight at least $c=c_{s,t,r}=|V(H)|$, since then we could find a copy of $H$. Indeed, if there are vertices $y_i$ for $1\leq i\leq t$ and $w_{j,k}$ for $1\leq j\leq s$, $1\leq k\leq r-1$ such that $N_G(\{y_i,w_{j,1},\dots,w_{j,r-1}\})$ contains at least $c$ elements for every $i,j$, then we can choose an element $z_{i,j}$ from each of these sets such that all the $x_i,y_j,z_{k,\ell}$ and $w_{a,b}$ are distinct, yielding a copy of $H$.
    Thus, as long as $|N(x_1)\cap \dots \cap N(x_{r-1})|\geq m$, we have that in $N(x_1)\cap \dots \cap N(x_{r-1})$ the proportion of those $r$-sets with weight at most $c$ is at least $\eta=\eta_{s,t,r}=1/{m \choose r}$. Since each $r$-set in $N_G(\{x_1,\dots,x_{r-1}\})$ is clearly an $r$-edge, it follows that the total number of $r$-edges of weight at most $c$ is at least 
    \begin{align}\label{eq:light_r_edgeweight}
    \frac{1}{{c \choose r-1}}\cdot \eta \cdot \sum_{\substack{x_1\dots x_{r-1}\in {V(G) \choose r-1} \\ d_G(x_1,\dots,x_{r-1})\geq m}} {d_G(x_1,\dots,x_{r-1}) \choose r},
    \end{align}
    where we used the fact that an $r$-tuple of weight at most $c$ is in at most $\binom{c}{r-1}$ of the sets $N_G(\{x_1,\dots,x_{r-1}\})$. Note now that 
    $$\sum_{x_1\dots x_{r-1}\in {V(G) \choose r-1}} d_G(x_1,\dots,x_{r-1})\geq n{\d \choose r-1}= \Omega(n\d^{r-1}).$$
    Therefore, on average $d_G(x_1,\dots,x_{r-1})$ is $\Omega(n(\d/n)^{r-1})=\omega(1)$, so, by Jensen's inequality, we have
    \begin{align*}
    \sum_{x_1\dots x_{r-1}\in {V(G) \choose r-1}} {d_G(x_1,\dots,x_{r-1}) \choose r} & \geq 2 {|V(G)| \choose r-1} {m \choose r} \\ 
    & \geq 2\sum_{\substack{x_1\dots x_{r-1}\in {V(G) \choose r-1} \\ d_G(x_1,\dots,x_{r-1})< m}} {d_G(x_1,\dots,x_{r-1}) \choose r}.
    \end{align*}
    Thus, together with~\eqref{eq:light_r_edgeweight}, the total number of $r$-edges of weight at most $c$ (and, therefore, the total weight of $r$-edges) is at least \begin{align}\label{eq:total_r_edgeweight}
        \frac{1}{2}\cdot\frac{1}{{c \choose r-1}}\cdot \eta \cdot \sum_{x_1\dots x_{r-1}\in {V(G) \choose r-1}} {d_G(x_1,\dots,x_{r-1}) \choose r}.
    \end{align}
    On the other hand, the total weight on $r$-edges of weight at least $L$ is at most 
    \begin{align}\label{eq:heavy_r_edgeweight}
        \frac{L}{{L \choose r-1}}\cdot \sum_{x_1\dots x_{r-1}\in {V(G) \choose r-1}} {d_G(x_1,\dots,x_{r-1}) \choose r},
    \end{align} 
    since an $r$-edge of weight $w$ is in ${w \choose r-1}$ of the sets $N_G(\{x_1,\dots,x_{r-1}\})$ and $w/{w\choose r-1}$ is a non-increasing function of $w$.
    If $r\geq 3$, then $L/{L\choose r-1}\rightarrow 0$ as $L\rightarrow\infty$ and, hence, the proportion of weight on heavy edges tends to $0$ as $L$ tends to infinity.
    
    In the $r=2$ case, \eqref{eq:heavy_r_edgeweight} does not help us,
    so we take a slightly different approach. For a constant $\varepsilon > 0$, let $\xi=\frac{\e\eta}{2c}$. If $N(x_1)$ contains more than $\xi\binom{d(x_1)}{2}$ pairs of weight at least $c$, then, for $n$ sufficiently large, there exists a copy of $H$. Indeed, the vertex $x_1$ together with a copy of $K_{s,t}$ in $N(x_1)$ formed by edges of weight at least $c$ easily extend to a nondegenerate copy of $H$. Thus, for large enough~$n$ and $L=c$, the total weight on edges of weight at least $L$ is at most
    \begin{align*}
        \xi \cdot \sum_{x\in V(G)} {d_G(x) \choose 2},
    \end{align*}
    which is at most $\e$ times \eqref{eq:total_r_edgeweight}.
\end{proof2}

\medskip

We remark that we in fact proved a slightly stronger statement than Lemma~\ref{rlight}. Indeed, the proof remains valid even if we replace $H$ by the supergraph obtained by adding additional edges between the $x_i$'s and $w_{j,k}$'s, since we embedded all $w_{j,k}$ into $N_G(\{x_1,\dots,x_{r-1}\})$.

\medskip

The following definition and lemma contain the key idea in our proof. Note that we continue to abuse notation slightly by referring to the vertices of $H_{s,t}(r)$ and their embedded images in another graph $G$ by the same labels.

\begin{definition}
An embedding of $H_{s,1}(r)$ in a graph $G$ is \emph{good} if the $r$-sets $\{x_1,\dots,x_{r-1},z_i\}$ and $\{y,w_{i,1},\dots,w_{i,r-1}\}$ are light in $G$ for every $1\leq i\leq s$.
\end{definition}

\begin{lemma} \label{goodh}
	Let $G$ be an $H$-free $K$-almost-regular graph on $n$ vertices with minimum degree~$\d=\omega(n^{1-\frac{1}{r-1}})$. Then, for $L$ sufficiently large, the number of good embeddings of $H_{s,1}(r)$ in $G$ is at least $\frac{1}{2}n\d^{sr+r-1}$.
\end{lemma}

\begin{proof2}
	The total weight on $r$-edges in $G$ is equal to the number of $r$-stars, which is at most $n(K\d)^r$ as $\Delta(G)\leq K\d$. Thus, Lemma \ref{rlight} implies that the number of $r$-stars whose leaf set is heavy is at most $c_Ln\d^r$, where $c_L\rightarrow 0$ as $L \rightarrow \infty$.
	
	Since $H_{s,1}(r)$ is a tree on $sr+r$ vertices and every vertex in $G$ has degree at least~$\d$, there are at least $(1-o(1))n\d^{sr+r-1}$ copies of $H_{s,1}(r)$ in $G$. By the first paragraph, $\{x_1,\dots,x_{r-1},z_1\}$ is heavy in at most $rc_Ln\d^r(K\d)^{sr-1}$ of them. Indeed, there are at most $(K\d)^{sr-1}$ ways to extend a fixed choice of $x_1,\dots,x_{r-1},y,z_1$, since $H_{s,1}(r)$ is connected and every vertex in~$G$ has degree at most $K\d$. The factor $r$ accounts for the fact that knowing the vertex set $\{x_1,\dots,x_{r-1},y,z_1\}$ of the $r$-star  leaves $r$ possibilities for $z_1$. The same holds for the other $r$-sets $\{x_1,\dots,x_{r-1},z_i\}$ and $\{y,w_{i,1},\dots,w_{i,r-1}\}$, so the number of copies of $H_{s,1}(r)$ which are not suitable is at most $2s\cdot rc_Ln\d^r(K\d)^{sr-1}=2rsc_LK^{sr-1}n\d^{sr+r-1}$. Since $c_L\rightarrow 0$ as $L \rightarrow \infty$, the result follows.
\end{proof2}

\vspace{3mm}

We are now in a position to prove Theorem~\ref{newexponents}.

\begin{proof}[\textnormal{\textbf{Proof of Theorem \ref{newexponents}}}]
	By Lemma \ref{goodh} and averaging, for sufficiently large $L$ there exist $x_i$ ($1\leq i\leq r-1$) and $w_{j,k}$ ($1 \leq j \leq s, 1 \leq k \leq r-1$) which extend to at least $\Omega(n^{1-(r-1)-s(r-1)}\d^{sr+r-1})=\omega(1)$ good embeddings of $H_{s,1}(r)$. Take a maximal set $\mathcal{M}$ of such extensions which are vertex-disjoint apart from the roots. If $\mathcal{M}$ consists of at least $t$ copies of $H_{s,1}(r)$, then their union forms a copy of $H_{s,t}(r)$.
    
    Suppose instead that $\mathcal{M}$ consists of at most $t-1$ extensions. Then any other extension has a non-root vertex which coincides with one of the non-root vertices of some $M\in \mathcal{M}$. Since there are $O(1)$ non-root vertices in the graphs $M\in \mathcal{M}$ and $O(1)$ vertices in $H_{s,1}(r)$,  there must exist some non-root vertex of $H_{s,1}(r)$ that is mapped to the same vertex in $\omega(1)$ of the good embeddings of $H_{s,1}(r)$ that extend $x_i$ ($1\leq i\leq r-1$) and $w_{j,k}$ ($1\leq j\leq s,1\leq k\leq r-1$). 
    
    Suppose first that $y$ is mapped to the same vertex in the $\omega(1)$ good copies of $H_{s,1}(r)$. Since $\{y,w_{j,1},\dots,w_{j,r-1}\}$ is light for every $j$, this leaves at most $O(1)$ possibilities for each~$z_j$, which contradicts the fact that our choice of $y$, $x_i$ and $w_{j,k}$ extend to $\omega(1)$ copies of $H_{s,1}(r)$. Similarly, suppose that some $z_j$ is mapped to the same vertex in $\omega(1)$ copies. Since $\{x_1,\dots,x_{r-1},z_j\}$ is light, this allows only $O(1)$ possibilities for $y$, which also leads to a contradiction.
\end{proof}

\section{Longer subdivisions of $K_{s,t}$} \label{sectionlongersub}

Recall that $L_{s,t}(k)$ is the $(k-1)$-subdivision of $K_{s,t}$ with an extra vertex joined to all vertices in the part of size $t$. This graph is the rooted $t$-blowup of $L_{s,1}(k)$, where $L_{s,1}(k)$ has vertices $u$, $v$, $w_{i,j}$ ($1\leq i\leq k, 1\leq j\leq s$) and edges $uv$, $vw_{1,j}$ ($1\leq j\leq s$), $w_{i,j}w_{i+1,j}$ ($1\leq i\leq k-1$, $1\leq j\leq s$), with roots $u,w_{k,1},\dots,w_{k,s}$.

In this section, we prove Theorem~\ref{longermain} and Corollary~\ref{longersubcor}. By Lemma~\ref{lemmaJS}, Theorem~\ref{longermain} reduces to the following.

\begin{theorem} \label{longersubdivisions}
	Let $s,t,k\geq 1$ be fixed integers and let $K\geq 1$ be a constant. Suppose that $G$ is a $K$-almost-regular graph on $n$ vertices with minimum degree $\d=\omega(n^{\frac{s}{sk+1}})$. Then, for $n$ sufficiently large, $G$ contains a copy of $L_{s,t}(k)$.
\end{theorem}

In what follows, let $s,t,k$ be fixed positive integers and let $K\geq 1$ be a constant. Let $H=L_{s,t}(k)$. As before, $L$ will be a constant to be determined in terms of $s$, $t$, $k$ and $K$, while $n$ will be sufficiently large compared to $s$, $t$, $k$, $K$ and $L$.

\begin{definition}
	Let $L$ be a positive integer. Define the function $f(\ell,L)$ for $1\leq \ell\leq k$ recursively by setting $f(1,L)=L$ and, for $2\leq \ell\leq k$,  $$f(\ell,L)=1+f(\ell-1,L)^{16}(\ell-1)^2\max_{1\leq i\leq \ell-1} f(i, L)f(\ell-i, L).$$
	Given this notation, we recursively define the notions of \emph{admissible} and \emph{good paths of length $\ell$} in a graph.
	Any path of length $1$ is both admissible and good. For $2\leq \ell\leq k$, we say a path $P=v_0v_1\dots v_\ell$ is admissible if every proper subpath of $P$ is good, i.e., $v_iv_{i+1}\dots v_j$ is good for every $(i,j)\neq (0,\ell)$. A path $P$ is good if it is admissible and the number of admissible paths of length $\ell$ between $v_0$ and $v_\ell$ is at most $f(\ell,L)$.
	In particular, a good path of length 2 connects two end vertices with at most $1+L^{18}$ common neighbours, which essentially means that the end vertices form a light edge in the sense of the definition given in Section~\ref{sectionsubbipnew} with suitably chosen parameters.
\end{definition}

The function $f(\ell, L)$ was defined so that the following lemma holds.

\begin{lemma} \label{manydisjoint}
	If a path $P=v_0\dots v_\ell$ is admissible, but not good, then there exist at least $f(\ell-1,L)^{16}$ pairwise internally vertex-disjoint admissible paths of length $\ell$ from $v_0$ to $v_\ell$.
\end{lemma}

\begin{proof2}
	Choose a maximal set of pairwise internally vertex-disjoint admissible paths of length $\ell$ from $v_0$ to $v_\ell$. Call them $Q_1,\dots,Q_r$ and assume that $r<f(\ell-1,L)^{16}$. Every admissible path of length $\ell$ from $v_0$ to $v_\ell$ meets one of the paths $Q_1,\dots,Q_r$ at some vertex other than $v_0$ and $v_\ell$. But $P$ is not good, so there are at least $f(\ell,L)$ such paths. By the pigeonhole principle, it follows that there exist a vertex $w$ and some $1\leq i\leq \ell-1$ such that there are at least $\frac{f(\ell,L)}{f(\ell-1,L)^{16}(\ell-1)^2}$ admissible paths $x_0x_1\dots x_\ell$ with $x_0=v_0, x_i=w, x_\ell=v_\ell$. But $\frac{f(\ell,L)}{f(\ell-1,L)^{16}(\ell-1)^2}>f(i, L)f(\ell-i, L)$, so either there are more than $f(i, L)$ good paths of length $i$ from $v_0$ to $w$ or there are more than $f(\ell-i, L)$ good paths of length $\ell-i$ from $w$ to $v_\ell$. In either case we contradict the definition of a good path.
\end{proof2}

\medskip

Theorem \ref{longersubdivisions} will follow fairly easily from the next lemma, which says that for large enough $L$ only a small proportion of all paths of length $k$ are not good.

\begin{lemma} \label{fewnotgood}
	Let $G$ be an $H$-free $K$-almost-regular graph on $n$ vertices with minimum degree $\d=\omega(1)$. Then the number of paths of length $k$ which are not good is at most $c_Ln\d^k$, where $c_L\rightarrow 0$ as $L\rightarrow \infty$.
\end{lemma}

Using this result, we may prove the analogue of Lemma~\ref{goodh} for this setting.

\begin{lemma} \label{goodl}
	Let $G$ be an $H$-free $K$-almost-regular graph on $n$ vertices with minimum degree $\d=\omega(1)$. Then, for $L$ sufficiently large, the number of copies of $L_{s,1}(k)$ for which the paths $uvw_{1,j}\dots w_{k-1,j}$ and $vw_{1,j}\dots w_{k,j}$ with $1\leq j\leq s$ are all good is at least $\frac{1}{2}n\d^{sk+1}$.
\end{lemma}

\begin{proof2}
	Since $L_{s,1}(k)$ is a tree on $sk+2$ vertices and every vertex in $G$ has degree at least $\d$, there are at least $(1-o(1))n\d^{sk+1}$ copies of $L_{s,1}(k)$ in $G$. By Lemma \ref{fewnotgood}, at most $2c_Ln\d^k(K\d)^{(s-1)k+1}$ of them contain not good paths labelled by $uvw_{1,1}\dots w_{k-1,1}$. Indeed, there are at most $(K\d)^{(s-1)k+1}$ ways to extend a fixed choice of $u,v,w_{1,1},\dots,w_{k-1,1}$, since $L_{s,1}(k)$ is connected and every vertex in $G$ has degree at most $K\d$. The factor 2 accounts for the fact that knowing the path $uvw_{1,1}\dots,w_{k-1,1}$ leaves two possibilities for $(u,v,w_{1,1},\dots,w_{k-1,1})$. The same holds for the other paths, so the number of copies of $L_{s,1}(k)$ which are not suitable is at most $2s\cdot 2c_Ln\d^k(K\d)^{(s-1)k+1}=4sc_LK^{(s-1)k+1}n\d^{sk+1}$. Since $c_L\rightarrow 0$ as $L\rightarrow \infty$, the result follows.
\end{proof2}

\medskip

Before proving Lemma~\ref{fewnotgood}, we show how to conclude the proof of Theorem~\ref{longersubdivisions}.

\begin{proof}[\textnormal{\textbf{Proof of Theorem \ref{longersubdivisions}}}]
	Suppose for contradiction that $G$ is $H$-free. By Lemma~\ref{goodl}, if $L$ is sufficiently large, then there are distinct vertices $u,w_{k,1},\dots,w_{k,s}$ such that the number of ways to extend them to a copy of $L_{s,1}(k)$ in $G$ is at least $\frac{1}{n^{s+1}}\cdot \frac{1}{2}n\d^{sk+1}=\omega(1)$. Suppose that no $t$ of these are pairwise vertex-disjoint apart from the roots $u,w_{k,1},\dots,w_{k,s}$. Then, as in the proof of Theorem \ref{newexponents}, either $v$ or some $w_{i,j}$ ($1\leq i\leq k-1$, $1\leq j\leq s$) is mapped to the same vertex at least $\omega(1)$ times.
	
	Suppose first that it is $v$. Then, since $vw_{1,j}\dots w_{k,j}$ is good for each $j$, it follows that in these copies of $L_{s,1}(k)$, each tuple $(w_{1,j},\dots,w_{k-1,j})$ can take at most $f(k,L)=O(1)$ values, which 
	contradicts the assumption that our choice of $u$, $v$ and $w_{k,j}$ extend to $\omega(1)$ copies of~$L_{s,1}(k)$.
	Suppose now that some $w_{i,j}$ ($1\leq i\leq k-1$, $1\leq j\leq s$) is mapped to the same vertex $\omega(1)$ times. Then, since $uvw_{1,j}\dots w_{i,j}$ is a good path, there are only $O(1)$ possibilities for $v$. However, as we have just seen, once $u,v,w_{k,1},\dots,w_{k,s}$ are fixed, there are only $O(1)$ ways to extend them to a copy of $L_{s,1}(k)$. Hence, the fixed embedding of $u$, $w_{k,1},\dots,w_{k,s}$ and $w_{i,j}$ only extends to $O(1)$ copies of $L_{s,1}(k)$, which is again a contradiction.
	Thus, there must be at least $t$ copies of $L_{s,1}(k)$ extending $u,w_{k,1},\dots,w_{k,s}$ which are vertex-disjoint apart from the roots. That is, $G$ contains a copy of $L_{s,t}(k)$.
\end{proof}

It remains to prove Lemma \ref{fewnotgood}. We will need the following definition.

\begin{definition}
	A pair of distinct vertices $\{x, y\}$ in $G$ is said to be $\ell$-\emph{bad} for some $2 \leq \ell \leq k$ if there are at least $f(\ell-1,L)^{16}$ internally vertex-disjoint admissible paths of length $\ell$ from $x$ to $y$. In particular, Lemma~\ref{manydisjoint} implies that if there is an admissible, but not good, path of length $\ell$ from $x$ to $y$, then $\{x,y\}$ is $\ell$-bad.
\end{definition}

In what follows, for $v\in V(G)$, we shall write $\Gamma_i(v)$ for the set of vertices $u\in V(G)$ for which there exists a path of length $i$ from $v$ to $u$. The next lemma will be used to show that in an $H$-free graph there cannot be many bad pairs between $N(v)=\Gamma_1(v)$ and $\Gamma_{\ell-1}(v)$. We will take a suitable $X\s N(v)$, $Y=\Gamma_{\ell-1}(v)$ and repeatedly apply the lemma to obtain subdivided $t$-stars. At the end, we piece these together to form a copy of $H$. To make sure that this is nondegenerate, the set $Z$ of vertices that we have already used will be avoided.

\begin{lemma} \label{findstructure}
	Let $2\leq \ell\leq k$ and $1\leq i\leq \ell$. Let $G$ be a $K$-almost-regular graph on $n$ vertices with minimum degree $\d=\omega(1)$. Let $X,Y,Z\s V(G)$ be such that $|X|=\omega(1),|Z|\leq L^{1/10},|Y|\geq \frac{\d^{\ell-1}}{f(\ell-1,L)^2}$ and, for any $x\in X$, the number of $y\in Y$ such that $(x,y)$ is $\ell$-bad is as at least $\frac{|Y|}{f(\ell-1,L)^2}$. Then, provided that $L$ is sufficiently large compared to $k$, $t$ and $K$, there exist an $(i-1)$-subdivided $t$-star in $G$, disjoint from $Z$, whose leaves form a set $R\s Y$, and a subset $X'\s X$ such that $|X'|=\omega(1)$ and $(x',r)$ is $\ell$-bad for every $x'\in X'$ and $r\in R$.
\end{lemma}

\begin{proof2}
	First note that we may assume $X\cap Z=\emptyset$. Let $Y'$ be the set of those $y\in Y$ for which the number of $x\in X$ such that $(x,y)$ is $\ell$-bad is at least $\frac{|X|}{2f(\ell-1,L)^2}$. Then $|Y'|\geq \frac{|Y|}{2f(\ell-1,L)^2}$ and the number of $(x,y)\in X\times Y'$ which are $\ell$-bad is at least $\frac{|X||Y'|}{2f(\ell-1,L)^2}\geq \frac{|X||Y|}{4f(\ell-1,L)^4}\geq \frac{|X|\d^{\ell-1}}{4f(\ell-1,L)^6}$. Thus, there are at least $\frac{|X|\d^{\ell-1}}{4f(\ell-1,L)^6}\cdot f(\ell-1,L)^{16}\geq |X|f(\ell-1,L)^9\d^{\ell-1}$ paths of length $\ell$ starting in $X$ and ending in $Y'$. In particular, there exists some $x^*\in X$ such that there are at least $f(\ell-1,L)^9\d^{\ell-1}$ paths starting at $x^*$ and ending in $Y'$.
    
    The number of such paths intersecting $Z$ is at most $|Z|\ell(K\d)^{\ell-1}$. Indeed, there are at most $|Z|$ choices for the element of $Z$ in the path, at most $\ell$ choices for its position in the path and, given a fixed choice for these, at most $(K\d)^{\ell-1}$ choices for the other $\ell-1$ vertices in the path. (Note that as $X\cap Z=\emptyset$, the vertex in $Z$ is not $x^*$.) But $|Z|\ell(K\d)^{\ell-1}\leq L^{1/10}\ell K^{\ell-1}\d^{\ell-1}$, so, for $L$ sufficiently large, using the fact that $f(\ell-1,L)\geq L$, there are at least $f(\ell-1,L)^8\d^{\ell-1}$ paths of length $\ell$ starting at $x^*$ and ending in $Y'$ that avoid $Z$. Moreover, there are at most $(K\d)^{\ell-i}$ different initial segments of length $\ell-i$ for these paths, so, by the pigeonhole principle, there exist $\frac{f(\ell-1,L)^8\d^{i-1}}{K^{\ell-i}}$ of them which start with the same $\ell-i$ edges. It follows that there exists some $u\in \Gamma_{\ell-i}(x^*)$ such that there are at least $\frac{f(\ell-1,L)^8\d^{i-1}}{K^{\ell-i}}$ paths of length $i$ from $u$ to $Y'$, all disjoint from $Z$.
    
    Take now a maximal set of such paths which are pairwise vertex-disjoint apart from at~$u$. We claim that there are at least $f(\ell-1,L)^7$ such paths. Suppose otherwise. Then all the $\frac{f(\ell-1,L)^8\d^{i-1}}{K^{\ell-i}}$ paths of length $i$ from $u$ to $Y'$ intersect a certain set of size at most $if(\ell-1,L)^7$ not containing $u$. But there are at most $(if(\ell-1,L)^7)\cdot i\cdot (K\d)^{i-1}$ such paths, which is a contradiction for $L$ sufficiently large.
    
    So we have $r\geq f(\ell-1,L)^7$ paths $P_1,\dots,P_r$ of length $i$ from $u$ to $Y'$ which are pairwise vertex-disjoint except at $u$ and avoid $Z$. Let the endpoints of these paths be $y_1,\dots,y_r$. Since $y_j\in Y'$ for all $j$, the number of pairs $(x,y_j)$ with $x\in X$ which are $\ell$-bad is at least $\frac{r|X|}{2f(\ell-1,L)^2}$. Therefore, by Jensen's inequality, on average an $x\in X$ has at least ${r/2f(\ell-1,L)^2 \choose t}$ $t$-sets $\{y_{j_1},\dots,y_{j_t}\}$ such that all $(x,y_{j_q})$ are $\ell$-bad. Since ${r/2f(\ell-1,L)^2 \choose t}\geq (\frac{1}{4f(\ell-1,L)^2})^t{r \choose t}$, there exists a $t$-set $\{y_{j_1},\dots,y_{j_t}\}\s \{y_1,\dots,y_r\}$ such that the set $$X'=\{x\in X: (x,y_{j_q}) \text{ is } \ell\text{-bad for all } 1\leq q\leq t\}$$ has size at least $|X|/(4f(\ell-1,L)^2)^t=\omega(1)$. We can now take $R=\{y_{j_1},\dots,y_{j_t}\}$ and the union of the paths $P_{j_1},\dots,P_{j_t}$ is a suitable $(i-1)$-subdivided $t$-star.
\end{proof2}

\medskip

We now iterate Lemma~\ref{findstructure}, as promised, to find a copy of $H$.

\begin{lemma}
	Let $G$ be an $H$-free $K$-almost-regular graph on $n$ vertices with minimum degree~$\d=\omega(1)$. Then, provided that $L$ is sufficiently large compared to $s$, $t$, $k$ and $K$, for any $2\leq \ell\leq k$ and any $v\in V(G)$, the number of admissible, but not good, paths of the form $v_0vv_2v_3\dots v_\ell$ is at most $\frac{2(K\d)^\ell}{f(\ell-1,L)}$.
\end{lemma}

\begin{proof2}
	Suppose otherwise. Let $Y=\Gamma_{\ell-1}(v)$. Suppose first that $|Y|<\frac{\d^{\ell-1}}{f(\ell-1,L)^2}$. Note that if the path $v_0vv_2\dots v_\ell$ is admissible, then $vv_2\dots v_\ell$ is good, so the number of admissible paths of length $\ell-1$ from $v$ to $v_\ell$ is at most $f(\ell-1,L)$. Hence, the number of admissible paths $uvv_2\dots v_\ell$ is at most $f(\ell-1,L)$ for any fixed $u$ and $v_\ell$. But then the number of admissible paths of the form $v_0vv_2\dots v_\ell$ is at most $|N(v)||Y|f(\ell-1,L)<K\d \frac{\d^{\ell-1}}{f(\ell-1,L)^2}f(\ell-1,L)<\frac{2(K\d)^\ell}{f(\ell-1,L)}$, which contradicts our assumption.
	
	We may therefore assume that $|Y|\geq \frac{\d^{\ell-1}}{f(\ell-1,L)^2}$. For any $x\in N(v)$ and any $y\in Y$, the number of admissible paths of the form $xvv_2\dots v_{\ell-1}y$ is again at most $f(\ell-1,L)$. Moreover, by assumption, the number of pairs $(x,y)\in N(v)\times Y$ such that there is an admissible, but not good, path of the form $xvv_2\dots v_{\ell-1}y$ is at least $\frac{2(K\d)^\ell}{f(\ell-1,L)^2}\geq \frac{2|N(v)||Y|}{f(\ell-1,L)^2}$. Recall that any such pair $(x,y)$ is $\ell$-bad. Let $X=\{x\in N(v): \text{ there are at least } \frac{|Y|}{f(\ell-1,L)^2} \text{ choices of } y\in Y \text{ for which } (x,y) \text{ is } \ell \text{-bad}\}$. Then $|X|\geq \frac{|N(v)|}{f(\ell-1,L)^2}\geq \frac{\d}{f(\ell-1,L)^2}=\omega(1)$.
	
	Our aim now is to find a copy of $H$ in $G$, which will yield a contradiction. Consider first the case $\ell=k$. By Lemma \ref{findstructure} with $Z=\{v\}$, there exists a set $X'\s X$ of size $\omega(1)$ and a set $R_1\s Y$ of size $t$ such that $v\not \in R_1$ and $(x,y)$ is $\ell$-bad for any $x\in X'$ and $y\in R_1$. Note that this uses Lemma~\ref{findstructure} in a rather weak sense since we do not need the subdivided star provided by the lemma, only its leaves. Now applying Lemma \ref{findstructure} with $Z=R_1\cup \{v\}$ and with $X'$ in place of $X$, we find a set $X''\s X'$ of size $\omega(1)$ and a set $R_2\s Y$ of size~$t$, disjoint from $R_1\cup \{v\}$ such that $(x,y)$ is $\ell$-bad for any $x\in X''$ and $y\in R_2$. Continuing like this, with a total of $\lceil \frac{s}{t}\rceil$ applications of Lemma~\ref{findstructure}, we can find a set $X_{\mathrm{final}}\s X$ of size $\omega(1)$ and a set $U= R_1\cup R_2\cup \dots \cup R_{\lceil s/t\rceil}\s Y$ with $|U|\geq s$ such that $X_{\mathrm{final}}$ and $U$ are disjoint and do not contain $v$ and, moreover, $(x,y)$ is $\ell$-bad for any $x\in X_{\mathrm{final}}$ and $y\in U$.  Choose distinct vertices $x_1,\dots,x_t\in X_{\mathrm{final}}$ and $y_1,\dots,y_s \in U$. Since $(x_i, y_j)$ is $\ell$-bad for every $i, j$, if $L$ is sufficiently large, we can find pairwise internally vertex-disjoint paths of length $\ell=k$ joining $x_i$ to $y_j$ for every $i,j$ and we can insist that these paths do not contain $v$. The union of these paths forms a copy of $K_{s,t}^{k-1}$. Together with the vertex $v$ and the edges $vx_1,\dots,vx_t$, we get a copy of $H$.
	
	\begin{figure}
\centering
\begin{tikzpicture}[scale=0.55]
\draw[fill=black](0,0)circle(5pt);
\draw[fill=black](6,1.5)circle(5pt);
\draw[fill=black](6,9)circle(5pt);
\draw[fill=black](6,12)circle(5pt);
\draw[fill=black](12,3)circle(5pt);
\draw[fill=black](12,6)circle(5pt);
\draw[fill=black](12,9)circle(5pt);
\draw[fill=black](12,12)circle(5pt);
\draw[fill=black](12,15)circle(5pt);
\draw[fill=black](12,18)circle(5pt);
\draw[fill=black](18,3)circle(5pt);
\draw[fill=black](18,6)circle(5pt);
\draw[fill=black](18,15)circle(5pt);
\draw[fill=black](18,18)circle(5pt);
\draw[fill=black](24,16.5)circle(5pt);

\draw[fill=black](6,-1.5)circle(5pt);
\draw[fill=black](6,-9)circle(5pt);
\draw[fill=black](6,-12)circle(5pt);
\draw[fill=black](12,-3)circle(5pt);
\draw[fill=black](12,-6)circle(5pt);
\draw[fill=black](12,-9)circle(5pt);
\draw[fill=black](12,-12)circle(5pt);
\draw[fill=black](12,-15)circle(5pt);
\draw[fill=black](12,-18)circle(5pt);
\draw[fill=black](18,-3)circle(5pt);
\draw[fill=black](18,-6)circle(5pt);
\draw[fill=black](18,-15)circle(5pt);
\draw[fill=black](18,-18)circle(5pt);
\draw[fill=black](24,-16.5)circle(5pt);

\draw[fill=black](6,0)circle(5pt);

\draw[fill=black](6,10.5)circle(5pt);
\draw[fill=black](12,4.5)circle(5pt);
\draw[fill=black](12,10.5)circle(5pt);
\draw[fill=black](12,16.5)circle(5pt);
\draw[fill=black](18,4.5)circle(5pt);
\draw[fill=black](18,16.5)circle(5pt);
\draw[fill=black](24,16.5)circle(5pt);

\draw[fill=black](6,-10.5)circle(5pt);
\draw[fill=black](12,-4.5)circle(5pt);
\draw[fill=black](12,-10.5)circle(5pt);
\draw[fill=black](12,-16.5)circle(5pt);
\draw[fill=black](18,-4.5)circle(5pt);
\draw[fill=black](18,-16.5)circle(5pt);
\draw[fill=black](24,-16.5)circle(5pt);

\draw[thick](0,0)--(6,-1.5)--(12,3)--(18,3)--(12,9)--(6,9)--(12,15)--(18,15)--(24,16.5);
\draw[thick,blue](0,0)--(6,1.5)--(12,6)--(18,6)--(12,12)--(6,12)--(12,18)--(18,18)--(24,16.5);

\draw[thick,blue](6,1.5)--(12,-3)--(18,-3)--(12,-9)--(6,-9)--(12,-15)--(18,-15)--(24,-16.5);
\draw[thick](6,-1.5)--(12,-6)--(18,-6)--(12,-12)--(6,-12)--(12,-18)--(18,-18)--(24,-16.5);

\draw[thick](0,0)--(6,0)--(12,4.5)--(18,4.5)--(12,10.5)--(6,10.5)--(12,16.5)--(18,16.5)--(24,16.5);
\draw[thick](0,0)--(6,0)--(12,-4.5)--(18,-4.5)--(12,-10.5)--(6,-10.5)--(12,-16.5)--(18,-16.5)--(24,-16.5);

\draw[rotate around={90:(6,0)},red] (6,0) ellipse (15 and 1);
\draw[rotate around={90:(18,0)},red] (18,0) ellipse (9 and 1);
\draw[rotate around={90:(21,16.5)},red] (21,16.5) ellipse (3 and 5);
\draw[rotate around={90:(21,-16.5)},red] (21,-16.5) ellipse (3 and 5);

\node at (0,-0.7) {$z$};

\node at (6,-2.2) {$w_0^3$};
\node at (12,2.3) {$w_{1,1}^3$};
\node at (18,2.3) {$w_{2,1}^3$}; 
\node at (12,8.3) {$w_{3,1}^3$};
\node at (6,8.3) {$w_{4,1}^3$};
\node at (12,14.3) {$w_{5,1}^3$};
\node at (18,14.3) {$w_{6,1}^3$};

\node at (6,-0.7) {$w_0^2$};
\node at (12,3.8) {$w_{1,1}^2$};
\node at (18,3.8) {$w_{2,1}^2$}; 
\node at (12,9.8) {$w_{3,1}^2$};
\node at (6,9.8) {$w_{4,1}^2$};
\node at (12,15.8) {$w_{5,1}^2$};
\node at (18,15.8) {$w_{6,1}^2$};

\node at (6,0.8) {$w_0^1$};
\node at (12,5.3) {$w_{1,1}^1$};
\node at (18,5.3) {$w_{2,1}^1$}; 
\node at (12,11.3) {$w_{3,1}^1$};
\node at (6,11.3) {$w_{4,1}^1$};
\node at (12,17.3) {$w_{5,1}^1$};
\node at (18,17.3) {$w_{6,1}^1$};

\node at (12,-6.7) {$w_{1,2}^3$};
\node at (18,-6.7) {$w_{2,2}^3$}; 
\node at (12,-12.7) {$w_{3,2}^3$};
\node at (6,-12.7) {$w_{4,2}^3$};
\node at (12,-18.7) {$w_{5,2}^3$};
\node at (18,-18.7) {$w_{6,2}^3$};

\node at (12,-5.2) {$w_{1,2}^2$};
\node at (18,-5.2) {$w_{2,2}^2$}; 
\node at (12,-11.2) {$w_{3,2}^2$};
\node at (6,-11.2) {$w_{4,2}^2$};
\node at (12,-17.2) {$w_{5,2}^2$};
\node at (18,-17.2) {$w_{6,2}^2$};

\node at (12,-3.7) {$w_{1,2}^1$};
\node at (18,-3.7) {$w_{2,2}^1$}; 
\node at (12,-9.7) {$w_{3,2}^1$};
\node at (6,-9.7) {$w_{4,2}^1$};
\node at (12,-15.7) {$w_{5,2}^1$};
\node at (18,-15.7) {$w_{6,2}^1$};

\node at (24,15.8) {$w_{7,1}$};
\node at (24,-17.2) {$w_{7,2}$};

\node[red] at (6,-15.6) {$X_{\mathrm{final}}$};
\node[red] at (18,-9.6) {$U$};
\node[red] at (21,13) {$T_1$};
\node[red] at (21,-20) {$T_2$};

\end{tikzpicture}

\caption{The embedding of $H=L_{s,t}(k)$ in the case $s=2$, $t=3$, $k=7$, $\ell=2$. The graph with the blue edges is $L_{s,1}(k)$, which is blown up $t$ times to give $H$.}

\label{fig:embedH}

\end{figure}
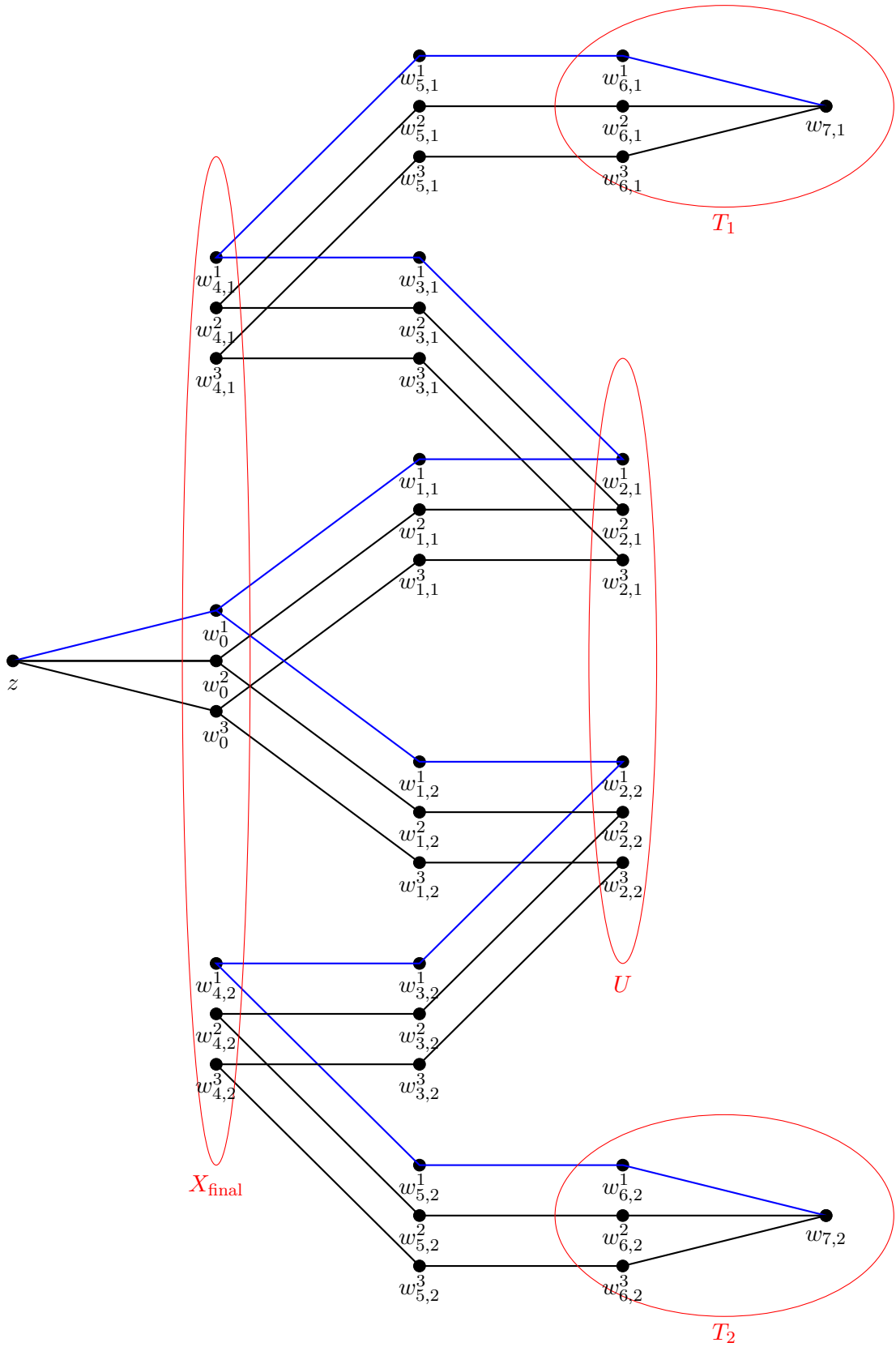
	
	Now assume that $\ell<k$. Write $k=j\ell+i$ with $1\leq i\leq \ell$. Note that $i<k$. Assume first that $j$ is odd. As in the case $\ell=k$, by repeated application of Lemma~\ref{findstructure}, we can find a set $X_{\mathrm{final}}\s X$ of size $\omega(1)$, $(i-1)$-subdivided $t$-stars $T_1,\dots,T_s$ with leaf sets $Y_1,\dots,Y_s\s Y$ and a set $U\s Y$ with $|U|=|V(H)|$ such that the sets $X_{\mathrm{final}},V(T_1),\dots,V(T_s),U$ are pairwise disjoint and do not contain $v$ and, moreover, $(x,y)$ is $\ell$-bad for any $x\in X_{\mathrm{final}}$ and $y\in Y_1\cup \dots \cup Y_s\cup U$. 
	
	At this point, we recall the definition of $H$. It is the $t$-blowup of the rooted tree $L_{s,1}(k)$ with vertices $z,w_0,w_{1,1},\dots,w_{1,s},w_{2,1},\dots,w_{2,s},\dots,w_{k,1},\dots,w_{k,s}$, roots $z,w_{k,1},\dots,w_{k,s}$ and edges $zw_0$, $w_0w_{1,b}$ ($1\leq b\leq s$) and $w_{a,b}w_{a+1,b}$ ($1\leq a\leq k-1$, $1\leq b\leq s$). 
	Let us see how we can find $H$ in $G$. 
	The $(i-1)$-subdivided $t$-star $T_1$ will take the role of the blowup of the path $w_{k-i,1}w_{k-i+1,1}\dots w_{k,1}$. More generally, $T_b$ ($1\leq b\leq s$) will take the role of the blowup of the path $w_{k-i,b}w_{k-i+1,b}\dots w_{k,b}$. Also, $v$ will take the role of $z$. Furthermore, the roles of the blown-up copies of $w_{a\ell,b}$ for $a$ odd ($1\leq a < j, 1\leq b\leq s$) will be taken by vertices in $U$ in an arbitrary injective manner and the roles of the blown-up copies of $w_{0}$ and $w_{a\ell,b}$ for $a$ even ($2\leq a < j, 1\leq b\leq s$) will be taken by vertices in $X_{\mathrm{final}}$ in an arbitrary injective manner. It remains to define the vertices that correspond to the blown-up copies of $w_{c,b}$ with $1\leq c\leq j\ell-1, 1\leq b\leq s$ and $c$ not divisible by $\ell$. For a vertex $u$ in $L_{s,1}(k)$, let $u^p$ denote the $p$th blownup copy of $u$. The vertices $w_{a\ell,b}^p,w_{(a+1)\ell,b}^p$ ($1\leq p\leq t, 0\leq a\leq j-1, 1\leq b\leq s$, where $w_{0,b}=w_0$) are embedded in $G$ in a way that one is in $X_{\mathrm{final}}$ and the other is in $Y_1\cup\dots\cup Y_s\cup U$, so the pair $(w_{a\ell,b}^p,w_{(a+1)\ell,b}^p)$ is $\ell$-bad in the embedding. Therefore, we may join these pairs by paths of length $\ell$, all disjoint from each other and from the previous vertices, yielding a nondegenerate copy of $H$. See Figure~\ref{fig:embedH}, which illustrates the embedding in the case $s=2$, $t=3$, $k=7$, $\ell=2$.
	
	The case where $j$ is even is very similar. The only difference is that we also need a $t$-star with leaf set $Q\s Y$, which is disjoint from all other sets and such that $(x,q)$ is $\ell$-bad for all $x\in X_{\mathrm{final}},q\in Q$. The existence of such a set again follows from Lemma~\ref{findstructure}. Then the role of the blowup of the edge $zw_0$ is taken by this $t$-star and the blown-up copies of $w_{a\ell,b}$ are chosen from $X_{\mathrm{final}}$ for $1\leq a<j$ odd and from $U$ for $2\leq a<j$ even.
\end{proof2}

\begin{corollary} \label{admissiblenotgood}
	Let $G$ be an $H$-free $K$-almost-regular graph on $n$ vertices with minimum degree~$\d=\omega(1)$. Then, provided that $L$ is sufficiently large compared to $s$, $t$, $k$ and $K$, for any $2\leq \ell\leq k$, the number of admissible, but not good, paths of length $\ell$ is at most $n\frac{2(K\d)^\ell}{f(\ell-1,L)}$.
\end{corollary}

It is now easy to deduce Lemma \ref{fewnotgood}.

\begin{proof}[\textnormal{\textbf{Proof of Lemma \ref{fewnotgood}}}]
	Suppose that the path $u_0u_1\dots u_k$ is not good. Take $0\leq i<j\leq k$ with $j-i$ minimal such that $u_iu_{i+1}\dots u_j$ is not good. Then $u_i\dots u_j$ is admissible. For any fixed $i,j$, by Corollary \ref{admissiblenotgood}, the number of such paths is at most $n\frac{2(K\d)^{j-i}}{f(j-i-1,L)}\cdot 2(K\d)^{k-(j-i)}=\frac{4K^k}{f(j-i-1,L)}n\d^k\leq \frac{4K^k}{L}n\d^k$. Using that $i$ and $j$ can take at most $k+1$ values each, it follows that the number of not good paths of length $k$ is at most $(k+1)^2\frac{4K^k}{L}n\d^k\leq c_Ln\d^k$, where $c_L\rightarrow 0$ as $L\rightarrow \infty$.
\end{proof}

Given Theorem \ref{longermain}, it is not hard to deduce Corollary \ref{longersubcor}. Recall that $L_{s,t}(k)$ is the rooted $t$-blowup of $L_{s,1}(k)$ with the roots defined as before. This rooted graph is balanced and bipartite with $\rho(L_{s,1}(k))=\frac{sk+1}{s(k-1)+1}$, so Lemma~\ref{lemmaBC} gives that $\ex(n,L_{s,t}(k)) = \Omega(n^{1 + \frac{s}{sk+1}})$ 
when $t$ is sufficiently large in terms of $s$ and $k$. Combining this with Theorem \ref{longermain}, Corollary \ref{longersubcor} follows.

For Proposition \ref{proplower}, note that $K_{s,t}^{k-1}$ is the rooted $t$-blowup of $K_{s,1}^{k-1}$ with the roots being the leaves of $K_{s,1}^{k-1}$. This rooted graph is balanced and bipartite with $\rho(K_{s,1}^{k-1})=\frac{sk}{s(k-1)+1}$. Thus, Lemma \ref{lemmaBC} implies that $\ex(n,K_{s,t}^{k-1}) = \Omega(n^{1 + \frac{s-1}{sk}})$ when $t$ is sufficiently large in terms of $s$ and $k$, as required.

\section{Concluding remarks}

{\bf More realisable numbers.} Following Kang, Kim and Liu~\cite{KKL18}, we say that a number $r \in (1,2)$ is \emph{balancedly realisable by a graph $F$} if there is a balanced connected rooted graph $F$ and a positive integer $\ell_0$ such that $\rho(F) = \frac{1}{2-r}$ and, for every $\ell \geq \ell_0$, the rooted $\ell$-blowup of $F$ has extremal number $\Theta(n^r)$. In their paper, Kang, Kim and Liu applied an old result of Erd\H{o}s and Simonovits~\cite{ES70} to prove the following useful lemma.

\begin{lemma}[Kang--Kim--Liu]
If $a$ and $b$ are positive integers with $b > a$ and $2 - \frac{a}{b}$ is balancedly realisable, then $2 - \frac{a}{a+b}$ is also balancedly realisable.
\end{lemma}

\noindent
Repeated applications of this lemma starting from the result~\cite{C18, FS83} that $1 + \frac{1}{a+1} = 2 - \frac{a}{a+1}$ is balancedly realisable for all $a$ then allowed them to show that $2 - \frac{a}{b}$ is balancedly realisable for all $b \equiv 1 \;(\bmod\; a)$. Applying the same reasoning starting from our Corollary~\ref{longersubcor} easily allows us to derive the following result.

\begin{corollary} \label{evenmoreexps}
For any integers $s, k, p \geq 1$, the exponent $2 - \frac{s k + 1}{p(sk+1) + s}$ is balancedly realisable. In particular, taking the limit as $s \rightarrow \infty$ implies that, for any positive integers $b > a$ with $b \equiv 1 \;(\bmod\; a)$, the exponent $2 - \frac{a}{b}$ is a limit point of the set of realisable numbers. 
\end{corollary}

\vspace{3mm}
\noindent
{\bf Subdivisions of complete bipartite graphs.}
Several interesting questions remain about subdivisions of complete bipartite graphs. One that immediately arises from Theorem~\ref{subbipnew} and Corollary~\ref{subbipnewcor} is the following.

\begin{problem} \label{smallest}
For any integer $s \geq 2$, estimate the smallest $t$ such that $\ex(n, K'_{s,t}) = \Omega(n^{3/2 - \frac{1}{2s}})$.
\end{problem}
For the analogous question with $K_{s,t}$ instead of $K'_{s,t}$, it has been conjectured~\cite{KST54} that $\ex(n, K_{s,t}) = \Omega(n^{2 - \frac{1}{s}})$ for all $t \geq s$, though this is only known for $s = 2$ or $3$ (see, for instance,~\cite{FS13}). The $s = 2$ case of Problem~\ref{smallest} amounts to estimating the extremal number of the theta graph $\theta_{4,t}$. Here it is known~\cite{VW19} that $\ex(n, \theta_{4,3}) = \Omega(n^{1 + 1/4})$. Deriving a similar bound for $\ex(n, \theta_{4,2})$ is likely to be difficult, as it would solve the famous open problem of estimating $\ex(n, C_8)$. However, the next case, when $s = 3$, now seems an attractive candidate for further exploration. 

Another pressing question is to improve the bound for $\ex(n, K_{s,t}^{k-1})$ given in Theorem~\ref{bipartitesubdivisions} so that it meets the lower bound given in Proposition~\ref{proplower}. This seems to be a good test case for developing methods that could help to resolve the full rational exponents conjecture.

\begin{conjecture}
For any integers $s, t, k \geq 1$, $\ex(n, K_{s,t}^{k-1}) = O(n^{1 + \frac{s-1}{sk}})$.
\end{conjecture}

\vspace{3mm}
\noindent
{\bf Hypergraph subdivisions.}
Given a hypergraph $\mathcal{H}$, its subdivision $\mathcal{H}'$ is defined to be the bipartite graph between $V(\mathcal{H})$ and $E(\mathcal{H})$ where we join $v \in V(\mathcal{H})$ and $e \in E(\mathcal{H})$ if and only if $v \in e$. In this language, Theorem~\ref{maxdegr} may be rephrased as saying that the subdivision $\L'$ of an $r$-uniform linear hypergraph $\L$ satisfies $\ex(n, \L') = o(n^{2 - 1/r})$. This is a special case of the following conjecture, itself a rather weak variant of Conjecture~\ref{conjecturefaks}.

\begin{conjecture}
For any $r$-uniform hypergraph $\mathcal{H}$, $\ex(n, \mathcal{H}') = o(n^{2 - 1/r})$.
\end{conjecture}
At present, we know this conjecture when $\mathcal{H} = K_{r+1}^{(r)}$, when $\mathcal{H}$ is $r$-partite~\cite{CL18} and, now, when $\mathcal{H}$ is linear. The methods of Section~\ref{sectionmaxdegr} also apply to some other hypergraphs for which the analogue of Theorem~\ref{embed} holds.
However, in full generality, the conjecture seems to lie well beyond our current methods, so any further progress would be extremely welcome.

\vspace{3mm}
\noindent
{\bf Acknowledgements.} We would like to thank the anonymous referees for their careful reviews.

\bibliographystyle{abbrv}
\bibliography{references}




\end{document}